%% file: main.tex
\documentclass[11pt]{article}
\usepackage[margin=1in]{geometry}
\usepackage[T1]{fontenc}
\usepackage{lmodern}
\usepackage{amssymb,bm,mathtools,amsthm}
\usepackage{enumitem}
\usepackage{algorithm}
\usepackage{algpseudocode}
\usepackage{graphicx}
\usepackage{booktabs}
\usepackage{placeins}
\usepackage{comment}
\usepackage{color}
\usepackage[hidelinks]{hyperref}
\hypersetup{
  pdftitle={Penalty-Uniform Localization for State-Constrained Policy Iteration},
  pdfauthor={Yeongjong Kim, Jiwoong Jang, Yeoneung Kim},
  pdfsubject={State-constrained optimal control and localized policy iteration},
  pdfkeywords={state constraints, Hamilton-Jacobi-Bellman equations, policy iteration, penalty approximation, physics-informed neural networks}
}
\newtheorem{theorem}{Theorem}[section]
\newtheorem{lemma}[theorem]{Lemma}
\newtheorem{proposition}[theorem]{Proposition}
\newtheorem{corollary}[theorem]{Corollary}
\newtheorem{definition}[theorem]{Definition}
\newtheorem{assumption}[theorem]{Assumption}
\theoremstyle{remark}
\newtheorem{remark}[theorem]{Remark}
\numberwithin{equation}{section}
\numberwithin{algorithm}{section}
\newenvironment{keywords}{\par\smallskip\noindent\small\textbf{Keywords.} }{\par}
\newenvironment{MSCcodes}{\par\smallskip\noindent\small\textbf{Mathematics Subject Classification.} }{\par\medskip}
\newcommand{\email}[1]{\href{mailto:#1}{#1}}

\newcommand{\R}{\mathbb{R}}

\newcommand{\dd}{\,\mathrm{d}}
\newcommand{\eps}{\varepsilon}
\newcommand{\argmin}{\operatorname*{arg\,min}}
\newcommand{\dist}{\mathrm{dist}}
\newcommand{\Lip}{\mathrm{Lip}}
\newcommand{\E}{\mathbb{E}}

\title{Penalty-Uniform Localization for State-Constrained Policy Iteration}
\hypersetup{pdftitle={Penalty-Uniform Localization for State-Constrained Policy Iteration},
pdfauthor={Yeongjong Kim, Jiwoong Jang, Yeoneung Kim}}
\author{Yeongjong Kim\thanks{Department of Mathematics, POSTECH.
Co-first author; equal contribution.}
\and Jiwoong Jang\thanks{Department of Mathematics, University of Maryland.
Co-first author; equal contribution.}
\and Yeoneung Kim\thanks{Corresponding author. Department of Industrial
Engineering, Yonsei University (\mbox{\email{yeoneung@yonsei.ac.kr}}).}}
\date{}

\begin{document}
\maketitle

\begin{abstract}
Penalizing a state constraint creates a singular localization problem:
whole-space values may grow like the inverse penalty parameter, while
numerical diffusion allows even an inward feedback to cross the boundary.
For deterministic discounted optimal control, we show how the penalty
itself supplies confinement that offsets this growth. With mesh size $h$
and penalty parameter $\varepsilon$, an inward barrier
bounds the additional cost of numerical leakage by $O(h/\varepsilon)$ for
a monotone centered-difference scheme with vanishing viscosity. Under
$h\le\varepsilon$, an occupation estimate then yields $O(h)$ localization
error with a sufficient box margin logarithmic in $1/h$ and independent
of $\varepsilon$. We extend the result to bounded squared-distance
penalties and combine it with residual-based evaluation bounds and
discounted policy-error propagation. Under coupled refinement with
vanishing penalization and discretization errors, explicit conditions on
evaluation errors and greedy gaps ensure convergence of the neural value
approximations to the constrained value, allowing measurable, nonunique
greedy selectors. Reference calculations isolate leakage and localization across
mesh and penalty scales, and distinguish evaluation, iteration, and
approximation errors. An explicit cylindrical state-constraint solution
provides a benchmark in arbitrary dimension, tested up to dimension twenty.
Paired obstacle-navigation experiments illustrate why policy-value accuracy
must be assessed alongside sampled residuals when comparing raw-residual
and finite-grid-assisted neural evaluation.
\end{abstract}


\section{Introduction}
\label{sec:intro}

State-constrained optimal control requires trajectories to remain in the
closure $\overline\Omega$ of a prescribed domain $\Omega$ for all times.
For deterministic problems with running cost $f$ and discount $\lambda>0$,
the value is a viscosity solution of the stationary
Hamilton--Jacobi--Bellman (HJB) equation in the domain and a viscosity
supersolution up to its boundary \cite{Soner1986,CapuzzoDolcettaLions1990}.
This inequality encodes the constraint without prescribed boundary data;
the resulting stability behavior differs from the Dirichlet setting
\cite{KimTranTu2020}. Alternative boundary formulations give uniqueness
and Lipschitz estimates \cite{IshiiKoike1996}.

The difficulty is the singular state-constraint limit.
For a bounded exterior penalty $p_\Omega/\varepsilon$ with scale
$\varepsilon>0$, whole-space values can be
$O(\varepsilon^{-1})$, so a localization bound proportional to their
global size deteriorates as $\varepsilon\downarrow0$. Moreover, the
artificial-viscosity chain can leave $\overline\Omega$ even under an
inward feedback: deterministic viability alone does not bound its
penalized cost uniformly. We quantify this leakage and use the penalty
itself to control exterior occupation. The resulting gain cancels the
inverse penalty factor in the exterior comparison, giving localization
constants uniform in $\varepsilon$ under coupled refinement. Classical
convergence of the penalized values \cite{CapuzzoDolcettaLions1990}
supplies the final link to the state-constraint solution.

State-constraint approximation has been studied through numerical convergence
rates \cite{CamilliFalcone1996}, feasible near-optimal feedbacks
\cite{IshiiKoike2000}, constraint relaxation and maximum-cost penalization
without controllability assumptions \cite{BokanowskiForcadelZidani2011},
and augmented-state epigraphs \cite{AltaroviciBokanowskiZidani2013}.
For vanishing viscosity, Han--Tu \cite{HanTu2022} obtain rates using viscous
solutions that diverge at the boundary. Here the integral exterior penalty
makes uniform control of localization the central issue.

Meng et al.\ \cite{MengEtAl2024} analyze neural policy iteration and
controller stability verification for undiscounted stabilization problems.
For monotone HJB schemes, geometric policy-iteration convergence and
discretization estimates are developed in \cite{TangTranZhang2025};
the stationary discounted iteration factor is given in
\cite[Theorem~5.2]{cho2026policy}.
Kim--Kim--Cho--Kim \cite{KimKimChoKim2026} develop the closely related
finite-horizon neural method: monotone shifted residuals, Markov-chain
localization with a shared exterior policy, residual-based evaluation
bounds, and greedy-gap propagation for measurable, possibly nonunique
improvements. Their Poisson exit estimates and time-integral error recursion
\cite[Theorems~5.3, 7.1, and~7.5]{KimKimChoKim2026} provide the localization
and evaluation framework. We quantify the singular penalty dependence:
numerical leakage and exterior occupation determine whether these tools
remain effective as the state constraint is enforced. The contributions
are as follows.
\begin{enumerate}[leftmargin=2em]
\item \emph{Leakage control and penalty-uniform localization.}
For mesh size $h$, an inward barrier gives the policy-cost bound
$\|f\|_\infty/\lambda+C_*h/\varepsilon$
(Lemma~\ref{pc:lem:leakage}), with $C_*$ independent of $h$ and $\varepsilon$.
For bounded-cost policies, the $O(\varepsilon)$ occupation outside a fixed
neighborhood of $\overline\Omega$ cancels the singular exterior-value scale, yielding
$O(h)$ localization error with a sufficient margin $O(\log(1/h))$
independent of $\varepsilon$ for $h\le\varepsilon$
(Theorem~\ref{pc:thm:localization}). The squared-distance extension covers
the experimental penalty.
\item \emph{Accuracy conditions for localized neural iteration.}
Larger-box residual bounds (Corollary~\ref{cor:closed_certificate}) and
the discounted recursion (Proposition~\ref{prop:assumption_free_api})
give mesh-dependent evaluation and greedy-gap conditions for convergence,
allowing measurable, nonunique greedy controls.
\item \emph{Explicit state-constraint benchmarks.}
We give an explicit cylindrical extension of the boundary-selection
benchmark in \cite{KimTranTu2020}, valid in arbitrary dimension, and verify
its control value directly, including at the corners. Reference calculations
examine leakage and localization across
penalty scales; paired obstacle tests relate evaluation accuracy to
feedback performance.
\end{enumerate}

Sections~\ref{sec:state_constraint}--\ref{sec:pinn} formulate the problem,
scheme, and neural implementation. Sections~\ref{sec:confinement}--\ref{sec:error_analysis}
establish localization and solution accuracy, followed by the experiments
in Section~\ref{sec:experiments}. Auxiliary estimates and full numerical
protocols are collected in Appendices~\ref{supp:policy_cost}--\ref{supp:paired_obstacle}.

\section{State constraints and penalty approximation}
\label{sec:state_constraint}
\subsection{State-constraint infinite-horizon optimal control}
Let $\Omega\subset \R^d$ be a bounded domain with $C^2$ boundary and let $A\subset\R^m$ be compact.
We use $|\cdot|$ for the Euclidean norm and
$\|z\|_{\infty,D}:=\sup_{x\in D}|z(x)|$ for the pointwise supremum;
the notation $L^\infty(D)$ below has this meaning when applied to values
or residuals. Unsubscripted suprema range over the full domain of the
function. Integrals and $L^2$ norms use Lebesgue measure, denoted by $|D|$
for a set $D$, unless a sampling density is specified.
We write $K\Subset D$ for a compact set contained in the interior of $D$,
$B_r(x)$ for the open Euclidean ball, and $\mathbf1_E$ for the indicator
of a set or event $E$. The notation $\Lip_x$ denotes the Lipschitz
constant in $x$, uniformly over the control variable.
Consider the deterministic controlled dynamics
\begin{equation}\label{eq:dynamics}
\dot X^{x,\alpha}(t)=b(X^{x,\alpha}(t),\alpha(t)),\qquad X^{x,\alpha}(0)=x\in\overline\Omega,
\end{equation}
where the control $\alpha:[0,\infty)\to A$ is measurable and
$b:\R^d\times A\to \R^d$ is continuous and Lipschitz in $x$ (uniformly in $a\in A$).
We impose the \emph{state constraint}
\[
X(t)\in \overline{\Omega}\qquad \forall t\ge 0.
\]
Accordingly, define the admissible controls from $x$ by
\begin{equation}\label{eq:admissible}
\mathcal A(x):=\bigl\{\alpha(\cdot): X^{x,\alpha}(t)\in\overline{\Omega}\ \forall t\ge 0 \bigr\}.
\end{equation}
Let $f:\R^d\times A\to \R$ be a continuous and bounded running cost function and fix discount factor $\lambda>0$.
The \emph{state-constraint value function} is
\begin{equation}\label{eq:value_sc}
u(x):=\inf_{\alpha\in\mathcal A(x)}\ \int_{0}^{\infty} e^{-\lambda t}\, f(X^{x,\alpha}(t),\alpha(t))\,\dd t,
\qquad x\in \overline{\Omega}.
\end{equation}

\subsection{Standing assumptions and the HJB equation}
Define the Hamiltonian
\begin{equation}\label{eq:Hamiltonian}
H(x,p):=\sup_{a\in A}\Bigl\{ - b(x,a)\cdot p-f(x,a)\Bigr\}.
\end{equation}
Formally, $u$ solves the stationary discounted HJB equation
\begin{equation}\label{eq:HJB_sc}
\lambda u(x)+H(x,Du(x))=0 \quad \text{in }\Omega,
\end{equation}
together with the \emph{state-constraint boundary condition}: $u$ must be a viscosity
\emph{supersolution} up to the boundary in the sense of Soner \cite{Soner1986}
and Capuzzo-Dolcetta--Lions \cite{CapuzzoDolcettaLions1990}.
Concretely, one requires:
\begin{itemize}[leftmargin=2em]
\item $u$ is a viscosity solution of \eqref{eq:HJB_sc} in $\Omega$;
\item $u$ is a viscosity supersolution of \eqref{eq:HJB_sc} on $\overline{\Omega}$ (interpreted as the boundary condition).
\end{itemize}

\begin{assumption}[Standing assumptions for the state-constraint problem]
\label{ass:standing_state_constraint}

\noindent

\begin{itemize}
\item Let $\Omega\subset\mathbb R^d$ be a bounded open set with $C^2$ boundary, and let
$n(x)$ denote the outward unit normal vector on $\partial\Omega$. Let $A$ be a
compact control set as above and let $\lambda>0$.

\item The controlled dynamics and running cost satisfy
\[
\|b\|_{L^{\infty}(\R^d\times A)} + \|f\|_{L^{\infty}(\R^d\times A)} < \infty\quad\text{and}\quad\mathrm{Lip}_x(b) + \mathrm{Lip}_x(f) < \infty.
\]
We further assume that
\[
\mathrm{Lip}_x(b) < \lambda.
\]
Let
\[
M_b:=\|b\|_{L^{\infty}(\R^d\times A)}.
\]

\item We assume the uniform inward controllability condition
\[
    \forall x\in\partial\Omega,\qquad
    \exists a_x\in A
    \quad\text{such that}\quad
    b(x,a_x)\cdot n(x)\le -\nu_0
\]
for some constant $\nu_0>0$. This ensures viability:
$\mathcal A(x)$ in \eqref{eq:admissible} is nonempty for every
$x\in\overline\Omega$.

\item The Hamiltonian defined by \eqref{eq:Hamiltonian} is coercive, that is,
\[
H(x,p)\to\infty\quad\text{as }|p|\to\infty\text{ uniformly in }x\in\overline{\Omega}.
\]

\item We normalize the running cost once and
for all by assuming that
\[
    f\ \ge\ 0 .
\]
Replacing $f$ by $f+\|f\|_\infty$ shifts every value function below by
$\|f\|_\infty/\lambda$. Greedy policies, greedy gaps,
and value differences are unchanged. This normalization makes the penalty
part of each policy value bounded by its total cost; all statements
translate back by the constant shift.
\end{itemize}

\end{assumption}

Higher-order inward conditions also yield feasible-trajectory estimates
and boundary regularity \cite{ColomboKhalilRampazzo2022}.
The uniform first-order condition assumed here supplies the distance
geometry and feedback used to control numerical boundary leakage in
Section~\ref{sec:confinement}.

\begin{remark}[Distance geometry for the leakage barrier]
\label{pc:ass:viability}
After decreasing $\nu_0$ if necessary, there exist constants
$r_1\in(0,1]$, $\nu_0\in(0,1]$, $C_\Omega\ge1$, a
function $\tilde d\in C^{1,1}(\R^d)$, and a bounded measurable feedback
$\bar\pi:\R^d\to A$ such that
\begin{enumerate}[label=(V\arabic*),leftmargin=3em]
\item $|\nabla\tilde d|\le1$ on $\R^d$ and
      $\|D^2\tilde d\|_{L^\infty}\le C_\Omega$;
\item $\tilde d\le0$ on $\overline\Omega$ and
      $\tilde d(x)\ge\tfrac12\min\{\dist(x,\overline\Omega),\,r_1\}$
      for $x\notin\overline\Omega$;
\item $b(x,\bar\pi(x))\cdot\nabla\tilde d(x)\le-\nu_0$
      whenever $|\tilde d(x)|\le r_1$.
\end{enumerate}
The signed distance $d_s$, positive outside $\overline\Omega$, is $C^2$
on a sufficiently small collar. Choose $r_1$ so that $|d_s|\le3r_1$
lies in this collar. Compose $d_s$ with a smooth nondecreasing saturation
that equals the identity on $[-r_1,r_1]$, has derivative in $[0,1]$,
and is constant beyond $\pm3r_1$, with saturated magnitudes greater than
$r_1$. Extending by these constants gives (V1)--(V2), and
$|\tilde d|\le r_1$ implies $\tilde d=d_s$.
Uniform inward controllability and continuity give a finite collar cover
on each member of which a fixed control is inward. Choosing the first
applicable control defines a Borel feedback satisfying (V3), with a
possibly smaller $\nu_0$. Outside the collar choose any fixed control.
\end{remark}

Under Assumption~\ref{ass:standing_state_constraint}, the following holds:

\begin{theorem}[{\cite[Theorem X.2]{CapuzzoDolcettaLions1990}}]
The value function \eqref{eq:value_sc} belongs to $C(\overline\Omega)$
and is the unique continuous viscosity solution of \eqref{eq:HJB_sc}
in $\Omega$ that is a viscosity supersolution on $\overline\Omega$.
\end{theorem}

\subsection{Unconstrained approximation by penalization}
\label{sec:penalty}

We approximate the constrained problem by an unconstrained problem on
$\R^d$ with a penalty outside $\overline\Omega$
\cite[(37)]{CapuzzoDolcettaLions1990}.

We fix a
truncation level $M_p>0$ and take the truncated distance penalty
\begin{align}\label{ass:P}
p_\Omega(x):=\min\bigl\{\operatorname{dist}(x,\overline\Omega),\,M_p\bigr\},
\qquad x\in\mathbb R^d,\tag{P}    
\end{align}
which is bounded and $1$-Lipschitz,
vanishes exactly on $\overline\Omega$, and satisfies
$p_\Omega(x)\ge c_\delta:=\min\{\delta,M_p\}>0$ whenever
$\operatorname{dist}(x,\overline\Omega)\ge\delta>0$. The truncation level is fixed along refinement sequences, ensuring a
uniform penalty bound. The experiments use the bounded squared-distance
extension proved in Proposition~\ref{prop:squared_penalty}. We refer to
the distance penalty above as \emph{(P)}.

For $\varepsilon>0$, define the unconstrained value function on $\R^d$:
\begin{equation}\label{eq:value_eps}
u^\varepsilon(x)
:=
\inf_{\alpha(\cdot)}
\int_{0}^{\infty} e^{-\lambda t}\Bigl(f(X^{x,\alpha}(t),\alpha(t))
+\tfrac{1}{\varepsilon}\,p_\Omega(X^{x,\alpha}(t))\Bigr)\,\dd t,
\qquad x\in \R^d,
\end{equation}
where $X^{x,\alpha}$ solves \eqref{eq:dynamics} without state restriction and $\alpha(\cdot)$ ranges over all measurable controls.
Then $u^\varepsilon$ solves
\begin{equation}\label{eq:HJB_eps}
\lambda u^\varepsilon(x) + H\bigl(x,Du^\varepsilon(x)\bigr) = \tfrac{1}{\varepsilon}p_\Omega(x)
\qquad\text{in }\R^d,
\end{equation}
in the viscosity sense. Continuous dependence of trajectories gives
$\Lip(u^\varepsilon)\le(\Lip_x(f)+\varepsilon^{-1})/
(\lambda-\Lip_x(b))$. The bounded uniformly continuous solution is
unique, and the following penalty convergence result applies.

\begin{theorem}[{\cite[Theorem VII.1 and the final paragraph of Section X]{CapuzzoDolcettaLions1990}}]
Under Assumption~\ref{ass:standing_state_constraint} and \eqref{ass:P},
the values $u^\varepsilon$ in \eqref{eq:value_eps} converge uniformly
on $\overline\Omega$ to $u$ in \eqref{eq:value_sc} as $\varepsilon\downarrow0$.
\end{theorem}

\begin{remark}
The convergence statement concerns values. At finite $\varepsilon$,
the penalty permits excursions and does not impose trajectory admissibility.
\end{remark}

\section{Monotone scheme and discounted representation}
\label{sec:PI}
\label{subsec:centered_fd_artificial_viscosity}
We use the centered viscous operator of
\cite{TangTranZhang2025,cho2026policy,KimKimChoKim2026}. Fix $N>0$
independently of $h$ and $\varepsilon$. For $h>0$, with $e_i$ the
$i$th coordinate vector, set
\begin{align*}
 (\nabla_hv)_i(x)&=\frac{v(x+he_i)-v(x-he_i)}{2h},\\
 \Delta_hv(x)&=\sum_{i=1}^d\frac{v(x+he_i)-2v(x)+v(x-he_i)}{h^2},
 \qquad\nu_h=Nh.
\end{align*}
All shifted values are whole-space evaluations. We impose the
nonnegative-neighbor condition
\begin{equation}\label{eq:monotonicity_condition_N}
 N\ge\tfrac12\sup_{x,a}\max_i|b_i(x,a)|.
\end{equation}
Policies and candidate values are Borel measurable throughout. For a
policy $\pi:\R^d\to A$, put
$c_\pi=f(\cdot,\pi)+p_\Omega/\eps$ and
$\mathcal A^\pi=\lambda-b(\cdot,\pi)\cdot\nabla_h-\nu_h\Delta_h$.
The frozen and Bellman residuals are
\begin{align}
 \mathcal L_{\eps,h}^\pi v&=\mathcal A^\pi v-c_\pi,
       \label{eq:semidiscrete_policy_operator}\\
 \mathcal F_{\eps,h}[v]&=\lambda v+H(x,\nabla_hv)
                 -\nu_h\Delta_hv-p_\Omega/\eps.
       \label{eq:semidiscrete_bellman_operator}
\end{align}
Their bounded zeros are denoted by $S_{\eps,h}(\pi)$ and
$v_{\eps,h}$, respectively.

\paragraph{Discounted representation}\label{pc:sec:chain}
The Markov-chain approximation framework is classical
\cite{KushnerDupuis2001}. The jump-chain construction in
\cite[Section~4]{KimKimChoKim2026}
has total rate $\Lambda_h=2dN/h$. Its stationary discounted resolvent
uses
\[
 \mu=\lambda+\Lambda_h,\qquad\gamma_h=\Lambda_h/\mu,
 \qquad q_i^\pm(x)=\frac{N\pm b_i(x,\pi(x))/2}{2dN}.
\]
The $q_i^\pm$ are nonnegative and sum to one. Let $Y_m$ be the chain
on $x+h\mathbb Z^d$ with these probabilities for jumps $\pm he_i$.
Write $\E_x^\pi$ for expectation with $Y_0=x$, and set
$P^\pi v(x)=\sum_i[q_i^+(x)v(x+he_i)+q_i^-(x)v(x-he_i)]$.
The fixed-point equation and
its discounted representation are
\begin{align}
 S_{\eps,h}(\pi)&=c_\pi/\mu+\gamma_hP^\pi S_{\eps,h}(\pi),
            \label{pc:eq:fixed_point_form}\\
 S_{\eps,h}(\pi)(x)&=\mu^{-1}\E_x^\pi
              \sum_{m\ge0}\gamma_h^m c_\pi(Y_m).
            \label{pc:eq:representation}
\end{align}
Put $M_\eps=(\|f\|_\infty+M_p/\eps)/\lambda$.
Each lattice coset evolves independently, which is why pointwise suprema
are required. The following representation and comparison estimates
control policy values and residual errors.
\input{markov_tools}

\paragraph{Policy iteration (PI)}\label{subsec:exact_semidiscrete_policy_iteration}
Exact evaluation alternates $w_n=S_{\eps,h}(\pi_n)$ with
$\pi_{n+1}=\pi_{w_n}$, where
\begin{equation}\label{eq:greedy_selector_for_value}
 \pi_v(x)\in\argmin_{a\in A}
       \{f(x,a)+b(x,a)\cdot\nabla_hv(x)\}.
\end{equation}
The state-only penalty cancels from the minimization. The exact iteration
has the factor $\gamma_h=2dN/(\lambda h+2dN)$ of
\cite[Theorem~5.2]{cho2026policy}.
\label{rem:effective_discount}
Since $(1-\gamma_h)^{-1}=1+2dN/(\lambda h)$, approximate evaluation
requires mesh-dependent error control; Section~\ref{sec:error_analysis}
records the discounted inexact extension used here.

\section{Physics-informed implementation}
\label{sec:pinn}
The bounded neural value $v_\theta:\mathbb R^d\to\mathbb R$, with parameter
vector $\theta$, is evaluated at all stencil shifts on the whole space.
We call the residual-based implementation physics-informed neural network
policy iteration (PINN-PI). Let $Q_L=[-L,L]^d$ contain
$\overline\Omega$ in its interior, assume $0<h<L$, and define
\[
 Q_L^h=[-L+h,L-h]^d,
 \qquad \partial_hQ_L=Q_L\setminus Q_L^h.
\]
Choose a larger training box with $Q_L\Subset Q_{L'}^h$.

\subsection{Policy evaluation and improvement}
\label{subsec:pinn_policy_evaluation}\label{subsec:pinn_policy_improvement}
Policy evaluation approximates the frozen-policy value
$S_{\varepsilon,h}(\pi_n)$ using a chosen evaluator. The raw evaluator
uses neural residual minimization
\cite{RaissiPerdikarisKarniadakis2019,SirignanoSpiliopoulos2018}
for the monotone operator, with pointwise residual and objective
\begin{equation}\label{eq:pinn_policy_eval_residual}
 \mathcal R_{\varepsilon,h}(\theta;\pi_n)
       =\mathcal L_{\varepsilon,h}^{\pi_n}v_\theta,
\end{equation}
\begin{equation}\label{eq:pinn_policy_eval_loss}
 \mathcal J_n(\theta)=\int_{Q_{L'}^h}
  |\mathcal R_{\varepsilon,h}(\theta;\pi_n)(x)|^2\rho_{L'}(x)\,dx
  +\mathcal J_{\rm reg}(\theta).
\end{equation}
Here $\rho_{L'}$ is a sampling density and $\mathcal J_{\rm reg}$ allows
optional regularization. Computation replaces the integral by a collocation
average. Since $\lambda>0$, the frozen equation has no additive-constant
ambiguity. In the paired obstacle experiment
(Section~\ref{subsec:obstacle_navigation}), the grid-assisted evaluator
fits computed finite-grid frozen-policy values with specified Dirichlet
data. The analysis concerns the resulting error relative to
$S_{\varepsilon,h}(\pi_n)$ and permits either evaluation procedure.
Wang et al.\ \cite{WangLiHeWang2022} study the dependence of HJB residual
stability on the loss norm; our evaluation bounds require pointwise
residual control.

On $Q_L^h$, exact improvement sets
\begin{equation}\label{eq:pinn_hard_greedy_improvement}
 \pi_{n+1}(x)\in\argmin_{a\in A}
       \{f(x,a)+b(x,a)\cdot\nabla_hv_{\theta_n}(x)\}.
\end{equation}
To describe inexact minimization, define
\begin{equation}\label{eq:pointwise_greedy_gap}
 \Gamma_{\varepsilon,h}(v,\pi)
       =\mathcal F_{\varepsilon,h}[v]-\mathcal L_{\varepsilon,h}^{\pi}v
       \ge0.
\end{equation}
An inexact improvement satisfies, for a nonnegative function $\eta_n$,
\begin{equation}\label{eq:inexact_policy_improvement_condition}
 0\le\Gamma_{\varepsilon,h}(v_{\theta_n},\pi_{n+1})(x)
       \le\eta_n(x),\qquad x\in Q_L^h.
\end{equation}
Exact minimization for an approximate value has $\eta_n=0$: its policy-value
error is distinct from failure to minimize the current Hamiltonian.

\begin{remark}[Common exterior policy]\label{rem:localized_improvement}
Every policy equals the same inward default $\bar\pi$ outside $Q_L^h$.
The region mask is applied in training, validation, and trajectory evaluation.
This defines the localized policy class $\Pi_L$ analyzed below.
\end{remark}
\begin{remark}[Nested boxes]\label{rem:nested_boxes}
Fix a compact target $K\subset\overline\Omega$. Solution accuracy is
measured on $K$, policies are improved on $Q_L^h$, and residuals are
validated on $Q_{L'}^h$, with $K\Subset Q_L\Subset Q_{L'}^h$.
The larger box controls evaluation error on all of $Q_L$ through
Corollary~\ref{cor:closed_certificate}. Its boundary-strip contribution is
controlled by a known network amplitude bound, or a validated strip estimate.
\end{remark}

\subsection{Algorithm and output pairing}\label{subsec:pinn_pi_algorithm}
\begin{algorithm}[t]
\caption{Physics-informed policy iteration for the penalized state-constraint problem}
\label{alg:pinn_pi_sc}
\begin{algorithmic}[1]
\Require penalty scale $\varepsilon$, mesh size $h$, artificial viscosity
$\nu_h=Nh$, improvement box $Q_L$ and training box $Q_{L'}$ with
$Q_L\Subset Q_{L'}^h$, default policy $\bar\pi$, initial
policy $\pi_0=\bar\pi$, evaluator, number of outer iterations $N_{\rm PI}$.
\For{$n=0,\dots,N_{\rm PI}-1$}
\State \textbf{Policy evaluation:} compute a neural approximation
$v_{\theta_n}$ to $S_{\varepsilon,h}(\pi_n)$ using the chosen evaluator
(Section~\ref{subsec:pinn_policy_evaluation}).
\State \textbf{Policy improvement:} compute $\pi_{n+1}$ by the exact hard greedy
rule \eqref{eq:pinn_hard_greedy_improvement} on $Q_L^h$, or by an inexact
greedy rule satisfying \eqref{eq:inexact_policy_improvement_condition}, and
set $\pi_{n+1}:=\bar\pi$ on $\R^d\setminus Q_L^h$.
\State \textbf{Validation:} evaluate the residual of $(v_{\theta_n},\pi_n)$
on an independent set in $Q_{L'}^h$ and the gap of
$(v_{\theta_n},\pi_{n+1})$ on $Q_L^h$; record sample diagnostics or
validated supremum bounds, as available.
\EndFor
\State \textbf{Output:} the value-policy pair
$(v_{\theta_{N_{\rm PI}-1}},\pi_{N_{\rm PI}-1})$, or the post-improvement
pair $(v_{\theta_{N_{\rm PI}-1}},\pi_{N_{\rm PI}})$; a value function
associated with $\pi_{N_{\rm PI}}$ requires one additional
policy-evaluation step.
\end{algorithmic}
\end{algorithm}
The saved pair $(v_{\theta_n},\pi_n)$ identifies the equation used in
evaluation. A post-improvement pair $(v_{\theta_n},\pi_{n+1})$ requires a
residual reevaluation: its computational greedy gap is zero under exact
improvement, but its evaluation residual need not be small.

\subsection{Direct Bellman comparison}\label{subsec:why_pinn_pi}
\label{subsec:direct_pinn_vs_pinn_pi}
Direct Bellman training minimizes
\[
\begin{aligned}
 \mathcal J_{\rm B}(\theta)
 ={}&\int_{Q_L^h}|\mathcal F_{\varepsilon,h}[v_\theta](x)|^2
                 \rho_{L'}(x)\,dx\\
 &+\int_{Q_{L'}^h\setminus Q_L^h}
       |\mathcal L_{\varepsilon,h}^{\bar\pi}v_\theta(x)|^2
                 \rho_{L'}(x)\,dx .
\end{aligned}
\]
Policy iteration alternates a linear equation in the value with pointwise
control minimization; both methods optimize nonlinear neural parameters.
The final-output bound in Section~\ref{subsec:final_bellman_certificate}
applies to either method, while the policy-iteration estimate tracks
evaluation errors and greedy gaps.

\section{Penalty-induced confinement and localization}
\label{sec:confinement}

The penalty confines policies with uniformly bounded cost. Numerical leakage
adds at most $O(h/\eps)$ to the cost of an inward feedback, making this
class nonempty under $h\le\eps$. We combine these facts with the common
exterior policy to bound the gap between the localized and whole-space
Bellman values.

\subsection{Discounted occupation and exit estimates}
\label{pc:sec:confinement}
Throughout Sections~\ref{sec:confinement}--\ref{sec:error_analysis}, assume
\eqref{ass:P}, Assumption~\ref{ass:standing_state_constraint}, and
\eqref{eq:monotonicity_condition_N}. The set $K\subset\overline\Omega$
is compact; geometric margins will be specified separately.

\begin{lemma}[Discounted exterior occupation]
\label{pc:lem:occupation}
Assume $f\ge0$ and Assumption~\ref{ass:standing_state_constraint}.  Let $\pi$ be a bounded
measurable feedback and $x_0\in\R^d$ with $S_{\eps,h}(\pi)(x_0)\le C_0$.
Then, for every $\delta\in(0,M_p]$,
\begin{equation}
\label{pc:eq:occupation}
    \frac1\mu\,
    \E_{x_0}^\pi\Bigl[
        \sum_{m=0}^\infty \gamma_h^m\,
        \mathbf 1_{\{\dist(Y_m,\overline\Omega)\ge\delta\}}
    \Bigr]
    \ \le\
    \frac{C_0\,\eps}{\delta}.
\end{equation}
\end{lemma}

\begin{proof}
Since $f\ge0$, the representation \eqref{pc:eq:representation} gives
\[
    \frac{1}{\eps\mu}\,
    \E_{x_0}^\pi\Bigl[\sum_m\gamma_h^m p_\Omega(Y_m)\Bigr]
    \le
    S_{\eps,h}(\pi)(x_0)\le C_0 .
\]
On $\{\dist(\cdot,\overline\Omega)\ge\delta\}$ one has
$p_\Omega\ge\min\{\delta,M_p\}=\delta$, and \eqref{pc:eq:occupation} follows.
\end{proof}

We now convert the occupation bound into an exit estimate.  The geometry is
encoded in the following margin condition.

\begin{assumption}[(M): margin between $\Omega$ and the improvement box]
\label{pc:ass:margin}
There are $\delta\in(0,M_p]$ and $D_L>0$ such that
\begin{equation}
\label{pc:eq:margin}
\begin{gathered}
    \bigl\{y\in Q_L:\ \dist(y,\R^d\setminus Q_L)\le D_L+h\bigr\}
    \subset
    \bigl\{\dist(\cdot\,,\overline\Omega)\ge\delta\bigr\},
    \\
    \dist(K,\R^d\setminus Q_L)>D_L+h .
\end{gathered}
\end{equation}
We also require
$\overline\Omega_\delta:=\{\dist(\cdot,\overline\Omega)\le\delta\}\subset Q_L$.
For $Q_L=[-L,L]^d$ one may take any fixed $\delta$ with
$\overline\Omega_\delta\Subset Q_L$
and, for $L$ sufficiently large,
$D_L:=L-\sup_{y\in\overline\Omega_\delta}\max_i|y_i|-2h>0$, so that
$D_L\to\infty$ as $L\to\infty$ with $\delta$ fixed.
\end{assumption}

\begin{lemma}[Discounted exit estimate]
\label{pc:lem:exit}
Assume $f\ge0$, Assumptions~\ref{ass:standing_state_constraint} and \ref{pc:ass:margin}, and
let $\pi$, $x\in K$ satisfy $S_{\eps,h}(\pi)(x)\le C_0$.  Let
$T_L:=\inf\{m\ge0:\ Y_m\notin Q_L\}$ and $m_D:=\lceil D_L/h\rceil$.  Then
\begin{equation}
\label{pc:eq:exit}
    \E_x^\pi\bigl[\gamma_h^{T_L}\mathbf 1_{\{T_L<\infty\}}\bigr]
    \ \le\
    \frac{\lambda\,\eps\,C_0}
         {\delta\,\bigl(\gamma_h^{-m_D}-1\bigr)} .
\end{equation}
Furthermore, if $\lambda h\le 2dN$, then
$\gamma_h^{-m_D}\ge e^{\lambda D_L/(4dN)}$ and consequently
\begin{equation}
\label{pc:eq:exit_exponential}
    \E_x^\pi\bigl[\gamma_h^{T_L}\mathbf 1_{\{T_L<\infty\}}\bigr]
    \ \le\
    \frac{\lambda\,\eps\,C_0}{\delta}\,
    \Bigl(e^{\lambda D_L/(4dN)}-1\Bigr)^{-1}
    \ \le\
    \frac{2\lambda\,\eps\,C_0}{\delta}\,
    e^{-\lambda D_L/(4dN)}
\end{equation}
whenever $e^{\lambda D_L/(4dN)}\ge2$.
\end{lemma}

\begin{proof}
Each chain step moves by exactly $h$ in one coordinate, so
$|Y_m-Y_{m'}|\le h|m-m'|$.  On $\{T_L<\infty\}$ and for
$T_L-m_D\le m<T_L$ we have $Y_m\in Q_L$ and
$\dist(Y_m,\R^d\setminus Q_L)\le|Y_m-Y_{T_L}|\le h\,m_D\le D_L+h$, whence
$\dist(Y_m,\overline\Omega)\ge\delta$ by \eqref{pc:eq:margin}.  Also
$T_L\ge m_D$ because $x\in K$ and
$\dist(K,\R^d\setminus Q_L)>D_L+h\ge h\,m_D$.  Therefore
\[
    \sum_{m=0}^\infty\gamma_h^m
    \mathbf 1_{\{\dist(Y_m,\overline\Omega)\ge\delta\}}
    \ \ge\
    \sum_{m=T_L-m_D}^{T_L-1}\gamma_h^m
    =
    \gamma_h^{T_L}\,
    \frac{\gamma_h^{-m_D}-1}{1-\gamma_h}
    \qquad\text{on }\{T_L<\infty\}.
\]
Taking $\E_x^\pi$ and combining with Lemma~\ref{pc:lem:occupation},
\[
    \E_x^\pi\bigl[\gamma_h^{T_L}\mathbf 1_{\{T_L<\infty\}}\bigr]\,
    \frac{\gamma_h^{-m_D}-1}{1-\gamma_h}
    \ \le\
    \mu\,\frac{C_0\eps}{\delta},
\]
and $\mu(1-\gamma_h)=\lambda$ gives \eqref{pc:eq:exit}.  For the exponential
form, with $t:=\lambda h/(2dN)\le1$,
\[
    \gamma_h^{-m_D}
    =(1+t)^{m_D}
    =e^{m_D\log(1+t)}
    \ge e^{m_D t/2}
    \ge e^{\lambda D_L/(4dN)},
\]
using $\log(1+t)\ge t/2$ for $t\in[0,1]$ and $m_D\ge D_L/h$.  Finally
$(e^s-1)^{-1}\le 2e^{-s}$ when $e^s\ge2$.
\end{proof}

\begin{remark}
\label{pc:rem:two_mechanisms}
Writing $d_{K,L}=\dist(K,\R^d\setminus Q_L)$, geometric discounting alone
bounds the exit weight by $\gamma_h^{\lceil d_{K,L}/h\rceil}$ for
\emph{arbitrary} policies: exponentially small in the margin but uniform in
$\eps$ and, crucially, \emph{not} vanishing as $h\downarrow0$ for a fixed
margin, with limit $e^{-\lambda d_{K,L}/(2dN)}$ when this distance is fixed.
Estimate \eqref{pc:eq:exit_exponential} shows that for
\emph{bounded-cost} policies the penalty contributes the extra factor
$\eps\,C_0/\delta$ on top of a comparable exponential rate: time at
distance at least $\delta$ costs at least $\delta/\eps$ per unit
discounted time, so trajectories that
reach $\partial Q_L$ are doubly suppressed.  This $\eps$ gain is exactly
what is needed to cancel the $O(\eps^{-1})$ magnitude of exterior penalized
values in Theorem~\ref{pc:thm:localization} below.
\end{remark}

\subsection{An inward barrier for numerical leakage}
\label{pc:sec:leakage}

The occupation argument requires a policy with uniformly bounded penalized
cost. Deterministic viability does not supply this for the artificial-viscosity
chain, which can cross $\partial\Omega$ under an inward feedback.
The following barrier controls this numerical leakage directly.

\begin{lemma}[Numerical boundary leakage]
\label{pc:lem:leakage}
Use the inward feedback of Remark~\ref{pc:ass:viability}. There exist
$C_*\ge1$ and $h_0\in(0,1]$, with $C_*$ depending only on
$(\lambda,\nu_0,d,N,M_b,C_\Omega)$ and $h_0$ additionally on $(M_p,r_1)$,
such that for every $h\in(0,h_0]$:
\begin{equation}
\label{pc:eq:leakage_penalty_value}
    w_{\bar\pi}(x)
    :=
    \frac1\mu\,\E_x^{\bar\pi}
    \Bigl[\sum_{m=0}^\infty\gamma_h^m\,p_\Omega(Y_m)\Bigr]
    \ \le\ C_*\,h
    \qquad\text{for all }x\in\overline\Omega .
\end{equation}
Consequently,
\begin{equation}
\label{pc:eq:leakage_value}
    S_{\eps,h}(\bar\pi)(x)
    \ \le\
    \frac{\|f\|_\infty}{\lambda}+\frac{C_*\,h}{\eps}
    \qquad\text{for all }x\in\overline\Omega .
\end{equation}
\end{lemma}

The barrier is constant away from the boundary on the interior side and
rises through a thin boundary layer, where the inward drift controls the
penalty. A cap preserves boundedness while retaining
the $O(h)$ scale on $\overline\Omega$.
\input{leakage_proof}

\begin{remark}[A sufficient structural regime]
\label{pc:rem:leakage_sharp}
The bound keeps the default-policy cost uniformly bounded for
$h\lesssim\eps$, including $h=\eps$; no vanishing ratio $h/\eps$ is
required. A matching lower bound is not claimed.
\end{remark}

\subsection{A common exterior policy}
\label{pc:sec:localized}
\begin{definition}[Localized policy class]\label{pc:def:localized_class}
For the inward default feedback $\bar\pi$, set
\[
 \Pi_L=\{\pi:\R^d\to A\text{ measurable}:\pi=\bar\pi
                   \text{ on }\R^d\setminus Q_L^h\},\qquad
 v_L^*=\inf_{\pi\in\Pi_L}S_{\eps,h}(\pi).
\]
Localized iteration starts from $\pi_0=\bar\pi$ and improves only on
$Q_L^h$. The value $v_L^*$ solves the Bellman equation there and the
default-policy equation elsewhere; a measurable localized greedy selector
attains the infimum.
\end{definition}
\input{localized_tools}

\subsection{Localization error of the Bellman solution}
\label{pc:sec:localization_theorem}

We now control the restriction from all policies to $\Pi_L$ in the singular
penalty limit. Throughout, $K\subset\overline\Omega$ is compact. Define
\[
 D_L^h=\dist(\overline\Omega_{\delta/2},\R^d\setminus Q_L^h),
 \qquad\overline\Omega_{\delta/2}
       =\{x:\dist(x,\overline\Omega)\le\delta/2\}.
\]
The exterior-cost magnitude is offset by the penalty gain in the discounted
exit estimate.

\begin{theorem}[Penalty-uniform localization]
\label{pc:thm:localization}
Assume \eqref{eq:monotonicity_condition_N}, \eqref{ass:P}, Assumptions
\ref{ass:standing_state_constraint} and~\ref{pc:ass:margin},
$h\le\min\{h_0,\eps\}$, $\eps\le1$, and $\lambda h\le 2dN$. Choose $h_0$ and the box
margin so that $D_L^h-2h>0$ and
$\exp(\lambda(D_L^h-2h)/(4dN))\ge2$. Set
\[
    \mathfrak g^{\rm ext}
    :=
    \sup_{\R^d\setminus Q_L^h}\Gamma_{\eps,h}(v_{\eps,h},\bar\pi),
    \qquad
    C_{\rm ext}:=\frac{2\sqrt d\,M_b\,(\|f\|_\infty+M_p)}{\lambda}
    +2\|f\|_\infty ,
\]
so that $\mathfrak g^{\rm ext}\le C_{\rm ext}/(\eps h)$.  If the margin
satisfies
\begin{equation}
\label{pc:eq:margin_condition}
    e^{-\lambda D_L/(4dN)}
    \ \le\
    \frac{\delta\,h^2}{24\,C_{\rm ext}} ,
\end{equation}
then
\begin{equation}
\label{pc:eq:localization_error}
    0
    \ \le\
    v^*_L(x)-v_{\eps,h}(x)
    \ \le\
    h\,\Bigl(\frac{\|f\|_\infty}{\lambda}+C_*\Bigr)
    \qquad\text{for all }x\in K .
\end{equation}
\end{theorem}

\input{localization_proof}

\begin{remark}[Uniformity and the sufficient margin]
\label{pc:rem:margin_reading}\leavevmode\par
The condition is
$D_L\ge(4dN/\lambda)\log(24C_{\rm ext}/(\delta h^2))$.
Its independence of $\eps$ comes from the occupation estimate; geometric
discounting alone would retain the singular penalty factor. The
$O(\log(1/h))$ margin is sufficient, not an optimal localization scale:
the proof uses the crude exterior gradient bound $O(1/(\eps h))$.
\end{remark}

\section{Solution accuracy under inexact evaluation}
\label{sec:error_analysis}
The confinement result closes the localization term in the approximation
$u\leftarrow u^\eps\leftarrow v_{\eps,h}\leftarrow v_L^*\leftarrow v_{\theta_n}$.
We combine it with discounted versions of the evaluation and greedy-gap
estimates in \cite[Section~7]{KimKimChoKim2026}.

\input{evaluation_certificates}

\subsection{Approximate policy iteration at
the localized fixed point}
\label{subsec:assumption_free}

The discounted fixed-point argument of \cite[Theorem~5.2]{cho2026policy}
extends to inexact improvement in the localized class $\Pi_L$. Comparing
each policy value directly with $v^*_L$ requires no stability of the greedy
selector.

\begin{proposition}[Discounted policy-error propagation]
\label{prop:assumption_free_api}
Let $\pi_n,\pi_{n+1}\in\Pi_L$, let $w_n=S_{\varepsilon,h}(\pi_n)$ and
$w_{n+1}=S_{\varepsilon,h}(\pi_{n+1})$, and recall
$g_n=\sup_{Q_L^h}\Gamma_{\varepsilon,h}(w_n,\pi_{n+1})$ and
$e_n:=\|w_n-v^*_L\|_{L^\infty(\mathbb R^d)}$
(note $w_n\ge v^*_L$). Then
\begin{equation}\label{eq:assumption_free_recursion}
    e_{n+1}\ \le\ \gamma_h\,e_n+\frac{g_n}{\lambda} .
\end{equation}
\end{proposition}

\input{policy_iteration_proof}

\begin{corollary}
\label{cor:assumption_free_total_error}
For every $n\ge1$,
\begin{equation}\label{eq:api_history_bound}
    e_n
    \le
    \gamma_h^n\,e_0
    +\frac1\lambda\sum_{j=0}^{n-1}\gamma_h^{\,n-1-j}
    \Bigl[
        \mathfrak G^{\infty}_{\varepsilon,h,L}(v_{\theta_j},\pi_{j+1})
        +\frac{2\sqrt d\,M_b}{h}\,\xi_j^{\rm val}
    \Bigr],
\end{equation}
and under exact hard improvement the computational gaps
$\mathfrak G^{\infty}_{\varepsilon,h,L}(v_{\theta_j},\pi_{j+1})$
vanish.
\end{corollary}

\begin{proof}
Iterate \eqref{eq:assumption_free_recursion} and bound $g_j$ by the
greedy-gap transfer estimate \eqref{eq:greedy_gap_transfer_applied}.
\end{proof}

\begin{remark}
By Lemma~\ref{pc:lem:no_exterior_maximum}, the $L^\infty(\mathbb R^d)$
and $L^\infty(Q_L)$ norms of $w_n-v^*_L$ coincide, so
\eqref{eq:assumption_free_recursion}--\eqref{eq:api_history_bound} may be
read in either norm.
\end{remark}

\begin{remark}
\label{rem:accuracy_requirement}
Since $(1-\gamma_h)^{-1}=O(h^{-1})$, \eqref{eq:api_history_bound}
accumulates evaluation and improvement errors as
$O(h^{-2}\sup_j\xi_j^{\rm val})$ and
$O(h^{-1}\sup_j\mathfrak G^{\infty}_{\varepsilon,h,L}
(v_{\theta_j},\pi_{j+1}))$, respectively. Uniform accuracies $o(h^2)$
and $o(h)$ therefore suffice under refinement. These worst-case sufficient
conditions link the mesh size to the evaluation accuracy needed by the
refinement theorem.
These are conditions on evaluation accuracy; convergence of the neural
optimizer is not assumed or proved.
\end{remark}

The residual-to-value statements above require a validated supremum.
A sufficient conversion from $L^2$ residuals, including its dimension-dependent
exponent and regularity requirements, is supplied in Appendix~\ref{supp:l2_lifting}.

\subsection{Total solution error}
\label{subsec:total_solution_error}

For compact $K\subset\overline\Omega$, define the penalty error
\begin{equation}\label{eq:penalty_error_K}
    \mathfrak P_K(\varepsilon)
    :=
    \|u-u^\varepsilon\|_{L^\infty(K)}.
\end{equation}
By the penalty approximation,
\[
    \mathfrak P_K(\varepsilon)\to0
    \qquad\text{as }\varepsilon\downarrow0.
\]
Define the discrete approximation error
\begin{equation}\label{eq:semidiscrete_error_K}
    \mathfrak D_K(\varepsilon,h,\nu_h)
    :=
    \|u^\varepsilon-v_{\varepsilon,h}\|_{L^\infty(K)}.
\end{equation}
This term contains the consistency error, the artificial-viscosity error, and
the interaction between the finite-difference scale and the penalized
boundary-layer structure.

\begin{remark}
\label{rem:semidiscrete_status}
For each fixed $\varepsilon>0$, the scheme
\eqref{eq:semidiscrete_bellman_operator} is monotone (under
\eqref{eq:monotonicity_condition_N}), uniformly stable, and consistent with
the penalized equation \eqref{eq:HJB_eps}, so
$v_{\varepsilon,h}\to u^\varepsilon$ locally uniformly as $h\downarrow0$
by the Barles--Souganidis framework \cite{BarlesSouganidis1991}.
The stability bound is $M_\varepsilon$, independent of $h$; whole-space
comparison for \eqref{eq:HJB_eps} identifies both half-relaxed limits
with $u^\varepsilon$. Thus
$\mathfrak D_K(\varepsilon,h,\nu_h)\to0$ for each fixed $\varepsilon$,
and condition (R5) of
Section~\ref{subsec:convergence_under_refinement} can be met by a diagonal
choice of $(\varepsilon_k,h_k)$. We do not quantify the rate; in
particular, the structural condition $h\le\varepsilon$ alone does not
imply the smallness of $\mathfrak D_K$.
\end{remark}

\begin{theorem}[Total solution error]
\label{thm:local_solution_error_mismatch}
Let $K\subset\overline\Omega\Subset Q_L$ be compact and assume the
hypotheses of Theorem~\ref{pc:thm:localization}. For localized policies
$\pi_j\in\Pi_L$ and their bounded neural evaluations, every $n\ge1$ satisfies
\begin{align}
    \|u-v_{\theta_n}\|_{L^\infty(K)}
    \le\;&
    \mathfrak P_K(\varepsilon)
    +\mathfrak D_K(\varepsilon,h,\nu_h)
    +h\Bigl(\frac{\|f\|_\infty}{\lambda}+C_*\Bigr)
    +\gamma_h^n\,\|w_0-v^*_L\|_{L^\infty(\mathbb R^d)}
    \notag\\
    &+\frac1\lambda
    \sum_{j=0}^{n-1}\gamma_h^{\,n-1-j}
    \Bigl[
        \mathfrak G^{\infty}_{\varepsilon,h,L}(v_{\theta_j},\pi_{j+1})
        +\frac{2\sqrt d\,M_b}{h}\,\xi_j^{\rm val}
    \Bigr]+\xi_n^{\rm val}.
    \label{eq:local_solution_error_mismatch}
\end{align}
Under exact hard improvement the computational gaps vanish. The estimate
requires no continuity or uniqueness of the greedy selectors.
\end{theorem}

\begin{proof}
Split
\[
    u-v_{\theta_n}
    =(u-u^\varepsilon)+(u^\varepsilon-v_{\varepsilon,h})
    +(v_{\varepsilon,h}-v^*_L)+(v^*_L-w_n)+(w_n-v_{\theta_n}).
\]
Taking norms, the first two terms give $\mathfrak P_K(\varepsilon)$ and
$\mathfrak D_K(\varepsilon,h,\nu_h)$; the third is bounded on $K$ by
Theorem~\ref{pc:thm:localization}; the fourth by
Corollary~\ref{cor:assumption_free_total_error}; the fifth by
$\xi_n^{\rm val}$ on $Q_L\supset K$.
\end{proof}

\subsection{Convergence under refinement}
\label{subsec:convergence_under_refinement}

For a refinement sequence, let $Q_{L_k}$ be the improvement boxes and let
$v_{\theta_{k,n_k}}$ denote the network obtained by evaluating
$\pi_{k,0},\ldots,\pi_{k,n_k}$, a total of $n_k+1$ evaluations.
Thus $n_k=N_{{\rm PI},k}-1$ in Algorithm~\ref{alg:pinn_pi_sc}.
Write $w_{k,j}=S_{\varepsilon_k,h_k}(\pi_{k,j})$ and
$\xi_{k,j}^{\rm val}=\|v_{\theta_{k,j}}-w_{k,j}\|_{\infty,Q_{L_k}}$;
$v^*_{L_k}$ is understood at $(\varepsilon_k,h_k)$.
Fix a compact target $K\subset\overline\Omega$.

\begin{corollary}
\label{cor:local_convergence_from_mismatch_control}
Let $\varepsilon_k\downarrow0$, $h_k\downarrow0$, $L_k\to\infty$,
$n_k\to\infty$, with $N$, $M_p$, and $\delta>0$ fixed.
Assume the standing hypotheses of this section, $\pi_{k,j}\in\Pi_{L_k}$,
and the margin geometry \eqref{pc:eq:margin} for each sufficiently large $k$.
Suppose also that:
\begin{enumerate}[label=(R\arabic*),leftmargin=3em]
\item $h_k\le\min\{h_0,\varepsilon_k\}$ and $\lambda h_k\le2dN$
      \hfill(structural regime; Remark~\ref{pc:rem:leakage_sharp});
\item $\displaystyle
      D_{L_k}\ \ge\ \frac{4dN}{\lambda}
      \log\Bigl(\frac{24\,C_{\rm ext}}{\delta\,h_k^2}\Bigr)$
      \hfill(margin \eqref{pc:eq:margin_condition});
\item $\gamma_{h_k}^{\,n_k}\,
      \|w_{k,0}-v^*_{L_k}\|_{L^\infty(\mathbb R^d)}\to0$,
      for which $\gamma_{h_k}^{\,n_k}/\varepsilon_k\to0$ is sufficient
      \hfill(initial-error decay);
\item $\displaystyle
      \frac{1}{h_k^{2}}\sup_{0\le j\le n_k}\xi_{k,j}^{\rm val}\to0$
      and
      $\displaystyle
      \frac{1}{h_k}\sup_{0\le j<n_k}
      \mathfrak G^{\infty}_{\varepsilon_k,h_k,L_k}
      (v_{\theta_{k,j}},\pi_{k,j+1})\to0$
      \hfill(evaluation and improvement accuracy);
\item $\mathfrak P_K(\varepsilon_k)\to0$ and
      $\mathfrak D_K(\varepsilon_k,h_k,\nu_{h_k})\to0$
      \hfill(penalization and discretization).
\end{enumerate}
Then
\[
    \|u-v_{\theta_{k,n_k}}\|_{L^\infty(K)}\to0
    \qquad\text{and}\qquad
    \|u-v_{\theta_{k,n_k}}\|_{L^2(K)}\to0 .
\]
\end{corollary}

\begin{proof}
By (R2), $D_{L_k}\to\infty$ and $D_{L_k}^{h_k}\ge D_{L_k}$, so the
additional reduced-box conditions of Theorem~\ref{pc:thm:localization}
hold for large $k$; also $\varepsilon_k\le1$ eventually.
Insert (R1)--(R5) into \eqref{eq:local_solution_error_mismatch}. The
localization term is $O(h_k)$. For (R3), the global value bound of
Lemma~\ref{pc:lem:representation} gives
\[
    \|w_{k,0}-v^*_{L_k}\|_{L^\infty(\mathbb R^d)}
    \le
    \frac{2\,(\|f\|_\infty+M_p/\varepsilon_k)}{\lambda}
    =O(\varepsilon_k^{-1}).
\]
Thus $\gamma_{h_k}^{n_k}/\varepsilon_k\to0$ suffices, and
$M_p$ is held fixed along the sequence. The weighted sum is
$O\bigl(h_k^{-2}\sup_{0\le j\le n_k}\xi^{\rm val}_{k,j}
+h_k^{-1}\sup_{0\le j<n_k}\mathfrak G^{\infty}_{\varepsilon_k,h_k,L_k}
(v_{\theta_{k,j}},\pi_{k,j+1})\bigr)\to0$ by (R4), using
$(1-\gamma_{h_k})^{-1}=O(h_k^{-1})$, and the final evaluation term
$\xi_{k,n_k}^{\rm val}\to0$ is covered by (R4), whose supremum extends to
$j=n_k$. The $L^2$ statement follows from
$\|w\|_{L^2(K)}\le|K|^{1/2}\|w\|_{L^\infty(K)}$.
\end{proof}

The nested-box version of this refinement statement and its conditional
$L^2$-residual formulation are given in Appendix~\ref{supp:refinement_certificates}.

\input{squared_penalty}

\input{final_residual}

\input{experiments}


\FloatBarrier
\bigskip
\appendix
\section*{Appendices}
\addcontentsline{toc}{section}{Appendices}
The appendices give auxiliary cost and residual results, conditional
$L^2$ formulations, the complete numerical protocol, and additional
diagnostics. The analytical statements use the bounded-data, monotonicity,
and pointwise-norm conventions of the main text; geometric hypotheses
are invoked where needed.

\input{appendices}

\FloatBarrier
\input{declarations}

\bibliographystyle{plainurl}
\phantomsection
\addcontentsline{toc}{section}{References}
{\small
\bibliography{ref}
}
\end{document}

%% file: markov_tools.tex
\begin{lemma}[Discounted representation]\label{pc:lem:representation}
Under \eqref{eq:monotonicity_condition_N}, every bounded measurable policy
$\pi$ has a unique bounded value given by \eqref{pc:eq:representation},
with $0\le S_{\eps,h}(\pi)\le M_\eps$.
\end{lemma}
\begin{proof}
Iteration of \eqref{pc:eq:fixed_point_form} gives
$v(x)=\mu^{-1}\E_x^\pi\sum_{m<M}\gamma_h^m c_\pi(Y_m)
+\gamma_h^M\E_x^\pi v(Y_M)$.
For bounded $v$ the remainder tends to zero. Conversely, the infinite
series is bounded by $\|c_\pi\|_\infty/[\mu(1-\gamma_h)]
=\|c_\pi\|_\infty/\lambda$ and satisfies the fixed-point equation
by conditioning on the first jump.
\end{proof}

\begin{lemma}[Comparison]\label{pc:lem:comparison}
If $z$ is bounded and $\mathcal A^\pi z\le g$ pointwise on $\R^d$,
where $\sup g<\infty$, then $\sup z\le\lambda^{-1}\sup g$.
\end{lemma}
\begin{proof}
Choose $x_\eta$ with $z(x_\eta)\ge\sup z-\eta$. Positivity of
$P^\pi$ gives $P^\pi z(x_\eta)-z(x_\eta)\le\eta$. Hence
$\lambda(\sup z-\eta)\le\lambda z(x_\eta)
\le\sup g+\Lambda_h\eta$. Let $\eta\downarrow0$.
\end{proof}
Applying comparison to $\pm z$ also yields
$\|z\|_\infty\le\lambda^{-1}\|\mathcal A^\pi z\|_\infty$.
These are the stationary discounted versions of the Markov and comparison
tools in \cite[Theorem~4.1]{KimKimChoKim2026}.
For the constant feedback $\pi\equiv a$, write $c_a$ and $P^a$.
The Bellman map obtained by minimizing $c_a/\mu+\gamma_hP^a v$
over $a\in A$ is also a monotone $\gamma_h$-contraction on bounded
Borel functions. Its unique fixed point lies in $[0,M_\eps]$.
A measurable greedy selector realizes the minimum, so this fixed point
is its policy value; comparison with every frozen-policy map shows that
it is the infimum of all policy values. The same argument applies when
the admissible control is fixed to $\bar\pi$ outside an improvement region.

%% file: leakage_proof.tex
\begin{proof}
By the discounted representation, $w=w_{\bar\pi}$ solves
$\mathcal A^{\bar\pi}w=p_\Omega$ and $0\le w\le M_p/\lambda$.
We construct a bounded supersolution of size $C_*h$ on $\overline\Omega$.

\emph{Barrier and active region.}
Put $C_5=5d(M_b+2N)$ and choose
\[
 C_*:=\max\{1,8C_5/(\lambda\nu_0)\},\qquad
 a=C_*h,\qquad \kappa=2(\nu_0a)^{-1/2}.
\]
Thus $\nu_0\kappa^2a=4$ and $C_5\kappa^2h\le\lambda/2$.
Define the $C^{1,1}$ profile
\[
 \psi(s)=\begin{cases}
 s,&s\ge-1/\kappa,\\
 -3/(2\kappa)+(\kappa/2)(s+2/\kappa)^2,
       &-2/\kappa\le s\le-1/\kappa,\\
 -3/(2\kappa),&s\le-2/\kappa,
 \end{cases}
\]
and set $\Theta=\psi\circ\tilde d$,
$\Phi_{\exp}=ae^{\kappa\Theta}$, and
\[
 \Phi=\min\{\Phi_{\exp},M_p/\lambda+a\}.
\]
We have $0\le\psi'\le1$, $|\psi''|\le\kappa$ a.e.,
$|\nabla\Theta|\le1$, and
$\|D^2\Theta\|_\infty\le\kappa+C_\Omega\le2\kappa$ for small $h$.
On $\overline\Omega$, $\Theta\le0$ and hence $\Phi\le a$.
The exponential branch is active only where
\[
 \tilde d\le s_*:=\kappa^{-1}\log(1+M_p/(\lambda a))
          =O(\sqrt h\log(1/h)).
\]
Choose $h_0\le1$ sufficiently small that $\kappa\ge C_\Omega$,
$2/\kappa\le r_1$, $s_*<r_1/2$, and $e^{2\kappa h}\le2$ for
$h\le h_0$. For active points outside $\overline\Omega$, (V2) then
implies $\dist(x,\overline\Omega)\le2\tilde d(x)\le2s_*<r_1$.
On the active set,
\begin{equation}\label{pc:eq:collar_inclusion}
 \psi'(\tilde d)>0\ \Longrightarrow\
 -2/\kappa<\tilde d\le s_*\ \Longrightarrow\
 b(x,\bar\pi(x))\cdot\nabla\tilde d(x)\le-\nu_0.
\end{equation}

\emph{Discrete supersolution estimate.}
At an active point write
$\Delta_i^\pm=\Theta(x\pm he_i)-\Theta(x)$.
The $C^{1,1}$ bound gives
$|\Delta_i^\pm\mp h\partial_i\Theta|\le\kappa h^2$ and
$|\Delta_i^\pm|\le2h$. Consequently
\[
 e^{\kappa\Delta_i^\pm}=1+\kappa\Delta_i^\pm+R_i^\pm,
 \qquad |R_i^\pm|\le4\kappa^2h^2.
\]
After division by $\Phi_{\exp}(x)$, the centered drift differs from
$-\kappa b_i\partial_i\Theta$ by at most $5M_b\kappa^2h$; the
viscosity term has absolute value at most $10N\kappa^2h$.
Indeed the linear Taylor errors in either the sum or difference of the
two increments are at most $2\kappa^2h^2$, and the two exponential
remainders contribute at most $8\kappa^2h^2$.
Summing over coordinates and using \eqref{pc:eq:collar_inclusion} yields
\begin{equation}\label{pc:eq:barrier_master}
 \mathcal A^{\bar\pi}\Phi_{\exp}
 \ge\Phi_{\exp}\bigl[\lambda+\nu_0\kappa\psi'(\tilde d)
                         -C_5\kappa^2h\bigr]
 \ge\Phi_{\exp}\bigl[\lambda/2+\nu_0\kappa\psi'(\tilde d)\bigr].
\end{equation}
Where $\psi'=0$, the drift contribution vanishes without an inward
condition. If $\tilde d\le0$, (V2) gives $p_\Omega=0$, so
\eqref{pc:eq:barrier_master} suffices. If $\tilde d>0$, then
$\Theta=\tilde d$, $\psi'=1$, and $e^s\ge s$ gives
\[
 \mathcal A^{\bar\pi}\Phi_{\exp}
 \ge a\nu_0\kappa e^{\kappa\tilde d}
 \ge a\nu_0\kappa^2\tilde d
 =4\tilde d\ge2\dist(x,\overline\Omega)\ge p_\Omega.
\]

\emph{Capping and comparison.}
The constant cap satisfies
$\mathcal A^{\bar\pi}(M_p/\lambda+a)=M_p+\lambda a\ge p_\Omega$.
At any point where a branch $\Phi_j$ is active,
$\Phi(x)=\Phi_j(x)$ and all neighboring values of $\Phi$ are at most
those of $\Phi_j$. Nonnegative stencil weights therefore give
$\mathcal A^{\bar\pi}\Phi(x)\ge\mathcal A^{\bar\pi}\Phi_j(x)$,
including ties. This local argument needs no global bound on the
uncapped branch. Thus the bounded function $\Phi$ is a supersolution
everywhere, and comparison gives $w\le\Phi\le C_*h$ on
$\overline\Omega$. Finally,
$S_{\eps,h}(\bar\pi)\le\|f\|_\infty/\lambda+\eps^{-1}w$
proves \eqref{pc:eq:leakage_value}.
\end{proof}

%% file: localized_tools.tex
This is the common-exterior construction of
\cite[Section~5.2]{KimKimChoKim2026}. The following consequences connect
local improvements to global value comparisons.

\begin{lemma}[Localized source]\label{pc:lem:no_exterior_source}
For $\pi,\pi'\in\Pi_L$,
\[
 \|S_{\eps,h}(\pi)-S_{\eps,h}(\pi')\|_\infty
 \le\lambda^{-1}\|\mathcal L^\pi_{\eps,h}S_{\eps,h}(\pi')\|_{\infty,Q_L^h}.
\]
\end{lemma}
\begin{proof}
For $z=S_{\eps,h}(\pi)-S_{\eps,h}(\pi')$,
$\mathcal A^\pi z=-\mathcal L^\pi_{\eps,h}S_{\eps,h}(\pi')$.
The source vanishes outside $Q_L^h$ because the policies agree there.
Apply Lemma~\ref{pc:lem:comparison} to $\pm z$.
\end{proof}

\begin{lemma}[No exterior maximum]\label{pc:lem:no_exterior_maximum}
For the same $z$, $\|z\|_{\infty,\R^d}=\|z\|_{\infty,Q_L}$.
This also holds for $z=S_{\eps,h}(\pi)-v_L^*$.
\end{lemma}
\begin{proof}
Outside $Q_L^h$, the homogeneous equation gives
$|z(x)|\le\gamma_h\|z\|_\infty$. If the global supremum exceeded
the one on $Q_L$, an approximating sequence outside $Q_L$ would give
$\|z\|_\infty\le\gamma_h\|z\|_\infty$, a contradiction.
For the last assertion use $v_L^*=S_{\eps,h}(\pi_L^*)$, where
$\pi_L^*\in\Pi_L$ is a localized optimal selector.
\end{proof}
Proposition~\ref{pc:prop:class_stability} also bounds the
costs along the iteration by the initial leakage cost plus accumulated
greedy gaps. The localization argument below obtains its comparison-policy
cost bound by absorption and does not require that auxiliary sum condition.

%% file: localization_proof.tex
\begin{proof}
\emph{Comparison policy and exterior source.}
Let $\pi^{\rm opt}$ be a measurable optimal selector for
$v_{\eps,h}$ and let $\widetilde\pi$ equal $\pi^{\rm opt}$ on
$Q_L^h$ and $\bar\pi$ elsewhere. Set
$\widetilde w=S_{\eps,h}(\widetilde\pi)$ and
$z=\widetilde w-v_{\eps,h}$.
Bellman verification gives $0\le v_L^*-v_{\eps,h}\le z$.
The frozen difference equation is
\[
 \mathcal A^{\widetilde\pi}z
 =\Gamma_{\eps,h}(v_{\eps,h},\bar\pi)
       \mathbf1_{\R^d\setminus Q_L^h}.
\]
The bound $\mathfrak g^{\rm ext}\le C_{\rm ext}/(\eps h)$ follows
from $\Gamma_{\eps,h}(v,\pi)\le2M_b|\nabla_hv|+2\|f\|_\infty$,
$|\nabla_hv|\le\sqrt d\|v\|_\infty/h$, and
$\|v_{\eps,h}\|_\infty\le(\|f\|_\infty+M_p/\eps)/\lambda$,
using $\eps,h\le1$.
For $T'=\inf\{m:Y_m\notin Q_L^h\}$, the discounted representation
and $\mu(1-\gamma_h)=\lambda$ yield
\begin{equation}\label{pc:eq:localization_source_exit}
 z(x)\le\frac{\mathfrak g^{\rm ext}}{\mu}
      \E_x^{\widetilde\pi}\sum_{m\ge T'}\gamma_h^m
 =\frac{\mathfrak g^{\rm ext}}{\lambda}
      \E_x^{\widetilde\pi}
         [\gamma_h^{T'}\mathbf1_{\{T'<\infty\}}].
\end{equation}

\emph{Confinement on the reduced box.}
We apply Lemma~\ref{pc:lem:exit} with threshold $\delta/2$ and
margin $D_L^h-2h$. To verify its geometry, points in
$\overline\Omega_{\delta/2}$ cannot belong to the margin strip in
\eqref{pc:eq:margin}, so their distance from $\R^d\setminus Q_L$
exceeds $D_L+h$. Removing the stencil strip costs at most $h$;
hence $D_L^h\ge D_L$. In $Q_L^h$, points at distance at most
$D_L^h-h$ from its complement lie outside
$\overline\Omega_{\delta/2}$, and
$\dist(K,\R^d\setminus Q_L^h)\ge D_L^h$ since
$K\subset\overline\Omega$. These are the reduced-box margin conditions.

The occupation bound is pointwise in the starting state, so it applies
with $C_0=\widetilde w(x)$; no uniform bound for the truncated policy
is assumed. Since $\lambda h\le2dN$,
\begin{align}
 \E_x^{\widetilde\pi}
       [\gamma_h^{T'}\mathbf1_{\{T'<\infty\}}]
 &\le \frac{4\lambda\eps\widetilde w(x)}{\delta}
       e^{-\lambda(D_L^h-2h)/(4dN)}\notag\\
 &\le \frac{12\lambda\eps\widetilde w(x)}{\delta}
       e^{-\lambda D_L/(4dN)},\label{pc:eq:localization_confined_exit}
 \end{align}
where $e^{\lambda h/(2dN)}\le e<3$.

\emph{Cancellation and absorption.}
Combining \eqref{pc:eq:localization_source_exit} and
\eqref{pc:eq:localization_confined_exit} gives
\[
 z(x)\le \alpha_{\rm loc}\,[v_{\eps,h}(x)+z(x)],\qquad
 \alpha_{\rm loc}:=\frac{12\eps\mathfrak g^{\rm ext}}{\delta}
          e^{-\lambda D_L/(4dN)}
 \le\frac{12C_{\rm ext}}{\delta h}e^{-\lambda D_L/(4dN)}.
\]
The factor $\eps$ from confinement cancels the inverse penalty scale
in $\mathfrak g^{\rm ext}$. Condition~\eqref{pc:eq:margin_condition}
gives $\alpha_{\rm loc}\le h/2\le1/2$, so
$z\le2\alpha_{\rm loc}v_{\eps,h}\le h v_{\eps,h}$.
Finally, Lemma~\ref{pc:lem:leakage} and $h\le\eps$ imply
\[
 v_{\eps,h}\le S_{\eps,h}(\bar\pi)
       \le\|f\|_\infty/\lambda+C_*h/\eps
       \le\|f\|_\infty/\lambda+C_*
       \quad\hbox{on }K.
\]
Together with $0\le v_L^*-v_{\eps,h}\le z$, this proves the claim.
\end{proof}

%% file: evaluation_certificates.tex
\subsection{Evaluation accuracy from validated residuals}
\label{subsec:main_evaluation_certificates}
Let $w_n=S_{\eps,h}(\pi_n)$ and
$\xi_n^{\rm val}=\|v_{\theta_n}-w_n\|_{\infty,Q_L}$.
For a value-policy pair write
\begin{align}
 \mathfrak E^\infty_{\eps,h,L}(v,\pi)
   &:=\|\mathcal L^\pi_{\eps,h}v\|_{\infty,Q_L^h},
       \label{eq:policy_eval_residual_linf}\\
 \mathfrak G^\infty_{\eps,h,L}(v,\pi)
   &:=\|\Gamma_{\eps,h}(v,\pi)\|_{\infty,Q_L^h}.
       \label{eq:greedy_gap_linf}
\end{align}
The stopped-chain estimate below is the discounted counterpart of
\cite[Theorem~7.1]{KimKimChoKim2026}.

\begin{proposition}[Interior evaluation bound]\label{prop:interior_certificate}
For a bounded candidate $v$, a bounded measurable policy $\pi$, and
$K\Subset Q_L^h$, set $d_K=\dist(K,\partial_hQ_L)$ and
$\beta_{h,L}=\gamma_h^{\lceil d_K/h\rceil}$. Then
\begin{equation}\label{eq:interior_certificate}
 \|v-S_{\eps,h}(\pi)\|_{\infty,K}
 \le\lambda^{-1}\mathfrak E^\infty_{\eps,h,L}(v,\pi)
   +\beta_{h,L}\|v-S_{\eps,h}(\pi)\|_{\infty,\partial_hQ_L}.
\end{equation}
If $\lambda h\le2dN$, then
$\beta_{h,L}\le e^{-\lambda d_K/(4dN)}$.
\end{proposition}
\begin{proof}
For $z=v-S_{\eps,h}(\pi)$ and $r=\mathcal L^\pi_{\eps,h}v$,
the fixed-point equation is $z=r/\mu+\gamma_hP^\pi z$.
Stop at $T=\inf\{m:Y_m\notin Q_L^h\}$. Boundedness and
$\gamma_h<1$ justify iteration to $T\wedge M$ followed by $M\to\infty$:
\[
 z(x)=\mu^{-1}\E_x^\pi\sum_{m<T}\gamma_h^m r(Y_m)
       +\E_x^\pi[\gamma_h^Tz(Y_T)\mathbf1_{\{T<\infty\}}].
\]
The source weights sum to at most $1/\lambda$. On exit,
$Y_T\in\partial_hQ_L$ because its predecessor lies in $Q_L^h$, and
$T\ge\lceil d_K/h\rceil$ because each jump has length $h$.
This proves the bound. The exponential form follows from
$\gamma_h^{-1}=1+\lambda h/(2dN)$ and
$\log(1+t)\ge t/2$ for $0\le t\le1$.
The same iteration gives the corresponding one-sided estimate when
$\mathcal A^\pi z\le r$, which will be used for Bellman comparison.
\end{proof}

The PI recursion needs accuracy on all of $Q_L$, including its stencil
strip. Validating on a larger box removes the unknown strip value.
\begin{corollary}[Closed evaluation bound]\label{cor:closed_certificate}
Assume $Q_L\Subset Q_{L'}^h$ and $\lambda h\le2dN$. With the outer
margin $d_{L,L'}=\dist(Q_L,\partial_hQ_{L'})$, one has
\begin{equation}\label{eq:closed_certificate}
 \xi_n^{\rm val}\le
 \lambda^{-1}\mathfrak E^\infty_{\eps,h,L'}(v_{\theta_n},\pi_n)
 +e^{-\lambda d_{L,L'}/(4dN)}
       (\|v_{\theta_n}\|_{\infty,\partial_hQ_{L'}}+M_\eps).
\end{equation}
\end{corollary}
\begin{proof}
Apply Proposition~\ref{prop:interior_certificate} in $Q_{L'}$ with
$K=Q_L$, and use $0\le w_n\le M_\eps$ on the outer strip.
\end{proof}
This supplies a sufficient condition for the evaluation requirement in
the refinement theorem. The residual supremum and strip amplitude must
be bounded independently of the sampled training loss. The validation
margin can depend on $\eps$ through $M_\eps$; its role differs from the
penalty-uniform improvement-box margin in
Theorem~\ref{pc:thm:localization}. Appendix~\ref{supp:l2_lifting} treats conversion from $L^2$ residuals under
additional regularity; that regularity is not automatic for bang-bang policies.

\subsection{Greedy-gap transfer}
The following Hamiltonian estimate is the same transfer mechanism as
\cite[Lemma~7.3 and Corollary~7.4]{KimKimChoKim2026}.
\begin{lemma}\label{lem:greedy_gap_transfer}
For bounded $v,w$ and any policy $\pi$,
\begin{equation}\label{eq:greedy_gap_transfer}
 \Gamma_{\eps,h}(w,\pi)(x)\le\Gamma_{\eps,h}(v,\pi)(x)
                         +2M_b|\nabla_h(v-w)(x)|.
\end{equation}
In particular, with $g_n=\sup_{Q_L^h}\Gamma_{\eps,h}(w_n,\pi_{n+1})$,
\begin{equation}\label{eq:greedy_gap_transfer_applied}
 g_n\le\mathfrak G^\infty_{\eps,h,L}(v_{\theta_n},\pi_{n+1})
                     +\frac{2\sqrt dM_b}{h}\xi_n^{\rm val}.
\end{equation}
\end{lemma}
\begin{proof}
Each controlled Hamiltonian and its supremum are $M_b$-Lipschitz in
the gradient; subtracting their values gives the factor $2M_b$.
For $x\in Q_L^h$, all stencil neighbors lie in $Q_L$, so
$|\nabla_h(v_{\theta_n}-w_n)(x)|\le\sqrt d\xi_n^{\rm val}/h$.
\end{proof}
Exact greedy improvement makes the first term in
\eqref{eq:greedy_gap_transfer_applied} zero, including at ties.
No continuity of the selected control is used; Remark~\ref{rem:ties} explains why comparing two intermediate policy
values would instead require additional control.

%% file: policy_iteration_proof.tex
\begin{proof}
Set $\mathcal T^\pi\phi=c_\pi/\mu+\gamma_hP^\pi\phi$, and let
$\mathcal T_L$ minimize this expression over controls on $Q_L^h$ and
use $\bar\pi$ elsewhere. Both maps are monotone $\gamma_h$-contractions
and satisfy $\mathcal T(\phi+c)=\mathcal T\phi+\gamma_hc$
for either map $\mathcal T$ and any constant $c$.
The fixed points of $\mathcal T^{\pi_n}$, $\mathcal T^{\pi_{n+1}}$,
and $\mathcal T_L$ are $w_n$, $w_{n+1}$, and $v_L^*$, respectively.
With $\delta_n=g_n/\mu$, the operator identities
\begin{equation}\label{eq:operator_identities}
 \mathcal L^\pi_{\eps,h}\phi=\mu(\phi-\mathcal T^\pi\phi),\qquad
 \Gamma_{\eps,h}(\phi,\pi)=\mu(\mathcal T^\pi\phi-\mathcal T_L\phi)
       \quad\text{on }Q_L^h
\end{equation}
give $0\le\mathcal T^{\pi_{n+1}}w_n-\mathcal T_Lw_n\le\delta_n$
globally: outside $Q_L^h$ the operator difference is zero.
Since $\mathcal T_Lw_n\le w_n$, the constant
$c_n=\delta_n/(1-\gamma_h)$ makes $w_n+c_n$ a supersolution for
$\mathcal T^{\pi_{n+1}}$. Monotone iteration gives
$w_{n+1}\le w_n+c_n$, and applying $\mathcal T^{\pi_{n+1}}$ once more,
\[
 w_{n+1}\le\mathcal T^{\pi_{n+1}}w_n+\gamma_hc_n
       \le\mathcal T_Lw_n+c_n.
\]
Using $v_L^*\le w_{n+1}$ and contraction yields
$0\le w_{n+1}-v_L^*\le\gamma_he_n+c_n$.
Finally $c_n=g_n/[\mu(1-\gamma_h)]=g_n/\lambda$.
\end{proof}

%% file: squared_penalty.tex
\subsection{Bounded squared-distance penalties}
\label{subsec:squared_penalty}

The experiments use a bounded squared-distance penalty. We specify its
penalization limit, leakage bound, and sufficient localization margin.
For a penalty $p$, write $u^p_\varepsilon$ for the continuous optimal value
and $S^p_{\varepsilon,h}(\pi)$ for a discrete policy value. Denote the
whole-space and localized Bellman values by $v^p_{\varepsilon,h}$ and
$v_L^{*,p}$, respectively, and set
$w_{\bar\pi}[p](x)=\mu^{-1}\E_x^{\bar\pi}\sum_{m\ge0}\gamma_h^m p(Y_m)$
for its discounted penalty value. In this subsection $C_*$ and $h_0$
are the leakage constants for the indicated truncated distance penalty.

\begin{proposition}[Squared-distance extension]
\label{prop:squared_penalty}
Fix $R>0$, independently of $h$ and $\varepsilon$, and set
\[
 p_R(x)=\min\{\dist(x,\overline\Omega),R\},
 \qquad q_R(x)=p_R(x)^2.
\]
Assume the standing hypotheses for the distance-penalty analysis and the
inward default policy of Lemma~\ref{pc:lem:leakage}; for the discrete
estimates take $0<h\le h_0$. Then $q_R$ is bounded
and Lipschitz, vanishes exactly on $\overline\Omega$, and satisfies:
\begin{enumerate}[label=(\roman*),leftmargin=2em]
\item The continuous penalized values with running cost $f+q_R/\varepsilon$
converge uniformly to the state-constraint value on $\overline\Omega$.
\item If $S^{q_R}_{\varepsilon,h}(\pi)(x)\le C_0$, then for every $\delta>0$,
\[
 \frac1\mu\E_x^\pi\sum_{m=0}^{\infty}\gamma_h^m
 \mathbf1_{\{\dist(Y_m,\overline\Omega)\ge\delta\}}
 \le \frac{\varepsilon C_0}{c_\delta},
 \qquad c_\delta=\min\{\delta,R\}^2.
\]
\item The default policy has penalty value
\[
 w_{\bar\pi}[q_R](x)\le RC_*h
 \quad(x\in\overline\Omega),
 \qquad
 S^{q_R}_{\varepsilon,h}(\bar\pi)(x)
 \le\frac{\|f\|_\infty}{\lambda}
       +\frac{RC_*h}{\varepsilon}.
\]
\end{enumerate}
For localization, take $0<\delta\le R$, assume the geometric conditions
of Assumption~\ref{pc:ass:margin}, and retain the small-parameter and reduced-box
conditions of Theorem~\ref{pc:thm:localization}, with $h_0$ for $p_R$.
Set
\[
 C_{\rm ext}^{(q)}
 :=\frac{2\sqrt d\,M_b(\|f\|_\infty+R^2)}{\lambda}
       +2\|f\|_\infty.
\]
Replacing \eqref{pc:eq:margin_condition} by
\begin{equation}\label{eq:squared_margin}
 e^{-\lambda D_L/(4dN)}
 \le\frac{c_{\delta/2}h^2}{12C_{\rm ext}^{(q)}}
\end{equation}
gives, for $x\in K$,
\[
 0\le v_L^{*,q_R}(x)-v^{q_R}_{\varepsilon,h}(x)
 \le h\left(\frac{\|f\|_\infty}{\lambda}+RC_*\right).
\]
The comparison, policy-iteration, and residual estimates use the penalty
bound $R^2$. The refinement results use this localization bound and
\eqref{eq:squared_margin} in place of their distance-penalty counterparts.
\end{proposition}

\begin{proof}
The map $q_R$ is bounded by $R^2$ and has Lipschitz constant at most $2R$.
For $\eta>0$, the elementary inequality $t^2\ge\eta t-\eta^2/4$ gives,
by integration along each admissible whole-space control and then taking
infima,
\[
 u^{p_R}_{\varepsilon/\eta}(x)
       -\frac{\eta^2}{4\lambda\varepsilon}
 \le u^{q_R}_{\varepsilon}(x)
 \le u^{p_R}_{\varepsilon/R}(x)
 \le u(x),\qquad x\in\overline\Omega.
\]
The last inequality uses constrained controls.  Choose
$\eta=\varepsilon^{3/4}$ and apply the distance-penalty convergence
theorem: both penalty parameters tend to zero, and the subtracted constant
is $\varepsilon^{1/2}/(4\lambda)$.  This proves (i) without assuming a
quantitative rate for the distance-penalty limit.

For (ii), nonnegativity of the running cost and $q_R\ge c_\delta$ on the
specified set give the estimate directly from the discounted Markov
representation.  For (iii), $0\le q_R\le Rp_R$ pointwise, so the same
representation under the fixed policy $\bar\pi$ gives
$w_{\bar\pi}[q_R]\le Rw_{\bar\pi}[p_R]\le RC_*h$.
The bound on the running-cost contribution is $\|f\|_\infty/\lambda$.
For localization, repeat the proof of Theorem~\ref{pc:thm:localization}
with the truncated comparison policy $\widetilde\pi$ and
$z=S^{q_R}_{\varepsilon,h}(\widetilde\pi)-v^{q_R}_{\varepsilon,h}$.
The exterior source is bounded by $C_{\rm ext}^{(q)}/(\varepsilon h)$.
Part (ii), at threshold $\delta/2$, gives a reduced-box exit weight at
most $6\lambda\varepsilon S^{q_R}_{\varepsilon,h}(\widetilde\pi)(x)
e^{-\lambda D_L/(4dN)}/c_{\delta/2}$. Hence
\[
 0\le v_L^{*,q_R}-v^{q_R}_{\varepsilon,h}\le z
 \le\frac{6C_{\rm ext}^{(q)}}{c_{\delta/2}h}
       e^{-\lambda D_L/(4dN)}(v^{q_R}_{\varepsilon,h}+z).
\]
Condition~\eqref{eq:squared_margin} makes the coefficient at most
$h/2\le1/2$. Absorption and part (iii) give the claimed bound.
The remaining estimates use only the global penalty bound and this
localization estimate, so the stated replacements suffice.
\end{proof}

\begin{remark}[A fixed penalty scale]
The implementation also allows
$q(x)=\min\{\kappa\dist(x,\overline\Omega)^2,R^2\}$ with a fixed
$\kappa>0$. This equals $\kappa p_r^2$ for $r=R/\sqrt\kappa$.
Its global bound is $R^2$, its lower bound at distance $\delta$ is
$c_\delta=\min\{\kappa\delta^2,R^2\}$, and its default-policy leakage is at most
$\kappa rC_*h=R\sqrt\kappa C_*h$.  Penalization convergence follows by
using $\varepsilon/\kappa$ in part (i). For localization use
\eqref{eq:squared_margin} with $c_{\delta/2}=\min\{\kappa\delta^2/4,R^2\}$,
$0<\delta\le r$, and the small-$h$ threshold for $p_r$.
When $h\le\varepsilon$, the default-policy
total cost is bounded by $\|f\|_\infty/\lambda+R\sqrt\kappa C_*$.
Thus increasing $\kappa$ strengthens the boundary penalty while its global
bound remains $R^2$. This fixed scaling is used for obstacle navigation.
\end{remark}

%% file: final_residual.tex
\subsection{Error of the final Bellman output}
\label{subsec:final_bellman_certificate}
The final-output bound below applies to any bounded candidate, including
one obtained by direct Bellman residual minimization.
Use the chosen penalty $p$ in all operators below, set
$M_\eps=(\|f\|_\infty+\|p\|_\infty/\eps)/\lambda$ and define
\[
 \mathcal R_L[v](x)=\begin{cases}
 \mathcal F_{\eps,h}[v](x),&x\in Q_L^h,\\
 \mathcal L^{\bar\pi}_{\eps,h}v(x),&x\notin Q_L^h.
 \end{cases}
\]
The following is the discounted localized counterpart of
\cite[Corollary~7.2]{KimKimChoKim2026}.
\begin{proposition}[Final-output residual bound]
\label{prop:final_bellman_certificate}
For a bounded candidate $v$, $L'>L$, and $K\Subset Q_{L'}^h$, put
$d_K=\dist(K,\partial_hQ_{L'})$ and
$\beta_K=\gamma_h^{\lceil d_K/h\rceil}$. Under monotonicity and bounded data,
\begin{equation}\label{eq:final_bellman_certificate}
 \|v-v_L^*\|_{\infty,K}
 \le\lambda^{-1}\|\mathcal R_L[v]\|_{\infty,Q_{L'}^h}
    +\beta_K(\|v\|_{\infty,\partial_hQ_{L'}}+M_\eps).
\end{equation}
The same estimate holds with all sets and suprema restricted to one lattice coset.
\end{proposition}
\begin{proof}
Choose localized greedy selectors $\pi^*$ for $v_L^*$ and $\pi_v$ for
$v$, both equal to $\bar\pi$ outside $Q_L^h$. For $z=v-v_L^*$,
the supremum structure of the Bellman residual gives
\[
 \mathcal A^{\pi^*}z\le\mathcal R_L[v],\qquad
 \mathcal A^{\pi_v}(-z)\le-\mathcal R_L[v].
\]
These inequalities hold on the exterior branch as well. Apply the stopped
comparison from Proposition~\ref{prop:interior_certificate} to each sign,
in the box $Q_{L'}$, and use $0\le v_L^*\le M_\eps$ on its strip.
\end{proof}
Adding the penalty, discretization, and localization errors bounds the
error against $u$ without accumulating intermediate evaluation errors.
The residual supremum must be validated as in
Section~\ref{subsec:main_evaluation_certificates}.

%% file: experiments.tex
\section{Numerical experiments}
\label{sec:experiments}

The experiments address three questions: how leakage and localization
depend on mesh and penalty scales; which error components remain when
the neural residual decreases; and how policy evaluation affects
obstacle navigation under matched neural training conditions. Explicit
solutions provide independent accuracy references, including a cylindrical
example in dimensions up to twenty.
Code, configurations, and recorded experimental artifacts are available at
\url{https://github.com/yeoneung/state_constraint_discrete}.

The neural runs and reference solves use the analyzed centered operator, its matching greedy
rule, and the common inward policy outside $Q_L^h$. The Bellman problems use
\[
 q(x)=\min\{\kappa\dist(x,\overline\Omega)^2,R^2\},\qquad R=0.3,
\]
as covered by Proposition~\ref{prop:squared_penalty}.
We use the viscosity multiplier $\rho\ge1/2$, setting
$N=\rho\sup_{x,a}\max_i|b_i(x,a)|$ and $\nu_h=Nh$, and
enforce $h\le\varepsilon$. Training and residual validation use the larger
box $Q_{L'}^h$ without clipping shifted network arguments.
All multilayer perceptrons (MLPs) have
three tanh hidden layers and bounded outputs; the reported stage is always
the last prescribed one. Analytic benchmark values and test trajectories never enter
training. Appendix~\ref{supp:numerical_protocol} gives the
complete sampling, precision, optimization and validation protocol.

\input{benchmarks}

\subsection{Confinement and resolved error diagnostics}
\label{subsec:reference_experiments}

For Benchmark II with $k=2$, $\rho=4$, and $\bar\pi(x)=-\tanh x$,
finite-lattice solves of $\mathcal A^{\bar\pi}\widehat w_r=r$ isolate
penalty-only leakage for $r=p,q$, where
$p=\min\{\dist(x,[-2,2]),0.3\}$ and $q$ uses $\kappa=1$.
For localization with $q/\varepsilon$, we test all 15 pairs $h\le\varepsilon$
from $\{0.04,0.02,0.01,0.005,0.0025\}$, including all five diagonal pairs
$h=\varepsilon$. The same rule $L(h)=2+0.06+0.02\log(0.04/h)$ is used at
every penalty scale.
Let $E_{\rm loc}$ be the maximum difference on domain nodes between
the localized and unrestricted Bellman values computed on $[-40,40]$.
\input{generated/confinement_results}
The separate coupled-refinement study also approaches the constrained
value: along $h=\varepsilon/2$, reducing $\varepsilon$ from $0.1$ to $0.003$
reduces the reference root-mean-square (RMS) errors in I and II from $4.95\times10^{-3}$ and
$4.30\times10^{-2}$ to $1.45\times10^{-4}$ and $1.72\times10^{-4}$,
respectively (Appendix~\ref{supp:mechanism_details}).

The neural diagnostics for I and II use $k=2$,
$(\varepsilon,h,\rho)=(0.03,0.015,0.55)$, and $(L,L')=(4,7)$.
An independent banded solver evaluates every saved neural policy and
computes the finite-box counterparts of $e_n$, $g_n$, and $e_{n+1}$ with
common zero Dirichlet data. The resulting conservative recursion bounds
are shown for all five seeds in Figure~\ref{fig:policy_recursion}.

To distinguish approximation effects, let $U_h$ be the finite-box Bellman
value, $W_n$ the frozen-policy value, and $U_{\rm fine}$ the Bellman value
on a nested finer mesh, all with the same penalty, boxes, and midpoint
Dirichlet data $M_\varepsilon/2$. On the same domain nodes, write
\begin{equation}\label{eq:numerical_decomposition}
 v_{\theta_n}-u
 =\underbrace{v_{\theta_n}-W_n}_{\text{evaluation}}
 +\underbrace{W_n-U_h}_{\text{policy iteration}}
 +\underbrace{U_h-U_{\rm fine}}_{\text{mesh/viscosity}}
 +\underbrace{U_{\rm fine}-u}_{\text{remaining bias}}.
\end{equation}
The final term retains penalty and localization effects and a small
unresolved discretization component; it is not an exact continuum penalty
error. The last mesh halving changes this proxy by less than $2\%$
in RMS (Appendix~\ref{supp:mechanism_details}). Norms of
the four signed fields need not add to the total error.

\input{generated/mechanism_results}
\begin{figure}[tbp]
\centering\includegraphics[width=\linewidth]{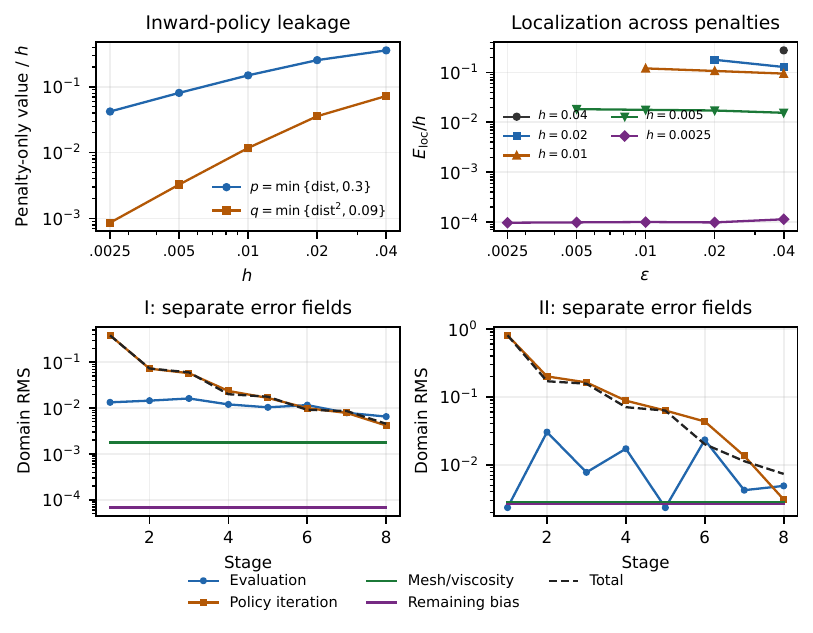}
\caption{Leakage and localization remain controlled across the tested
scales. Top left: $\|\widehat w_r\|_{\infty,\overline\Omega}/h$ for
the two penalties. Top right: localization error divided by $h$ for
all declared $h\le\varepsilon$ pairs with the same $L(h)$.
Bottom: means over five seeds of the separate field RMS norms in
\eqref{eq:numerical_decomposition}, with midpoint Dirichlet data.
All panels are finite-lattice diagnostics.}
\label{fig:error_mechanism}
\end{figure}

\subsection{Accuracy on explicit solutions}
\label{subsec:neural_exact_results}

Table~\ref{tab:exact_compact} compares PINN-PI with direct Bellman residual
minimization using paired initializations, architectures, samples and
nominal optimization budgets. Relative RMS divides the RMS error by the
RMS reference value, using $10{,}000$ equally spaced points in one dimension
and $20{,}000$ uniformly sampled domain points for the cylinder. These
full-domain tests include transverse directions. The cylindrical runs use
$k=1$ and $\rho=0.55$, with the remaining settings in Table~\ref{tab:aligned_protocol}. PINN-PI is more accurate
on average in dimensions five and ten; the advantage varies with dimension.
\input{generated/exact_compact}

The twenty-dimensional PINN-PI error remains about $16\%$. Its Bellman
residual decreases while the reference error stagnates and then increases
(Figure~\ref{fig:aligned_learning}). This illustrates why
a small sampled residual for a fixed approximate operator alone does not
establish accuracy for the constrained value. The explicit cylindrical
value makes this discrepancy observable throughout the domain.

\subsection{A paired obstacle-evaluation experiment}
\label{subsec:obstacle_navigation}

The free region is a disk of radius $1.2$ with five circular holes;
Figure~\ref{fig:paired_obstacle} shows the geometry. With $|a|\le1$ and
goal $x_g$, use
\[
 \begin{gathered}
 b(x,a)=1.25a,\quad \lambda=0.5,\quad x_g=(0.95,0.55),\\
 f(x,a)=\min\{c_g|x-x_g|^2,5\}+0.02|a|^2,\qquad c_g=385/512.
 \end{gathered}
\]
The cost cap is inactive in the free region. Set $\varepsilon=h=0.01$,
$\rho=2$, $\kappa=25$ and $(L,L')=(1.4,1.7)$. No safety filter is used.

For each of three seeds we compare two policy evaluators: raw frozen-policy
residual minimization and fitting the frozen-policy value computed on a
$341\times341$ grid with midpoint Dirichlet data. Both use identical
initial weights, training sample streams, width $256$, Adam batch size $4096$,
precision, and ten stages of $5000$ Adam updates followed by at most
$300$ L-BFGS iterations. The grid-assisted objective includes its finite-box
boundary treatment; it uses computed policy values, not optimal-value labels.
All policies are improved using the network's own centered gradient.

Finite-grid diagnostics compare each network with its frozen-policy value,
which also supplies the grid-assisted training targets. A separate solve
evaluates the final deployed greedy policy.
The protocol and seeds were fixed before training; $200$ common held-out test
starts are checked at three Euler step sizes $\Delta t$ over horizon $20$, stopping at
first entry into the goal ball of radius $0.15$. Every trajectory segment
is checked against all obstacles. Table~\ref{tab:paired_obstacle} reports all
six final controllers. Raw-arm failures stop near the goal but outside
radius $0.15$, at terminal distances about $0.18$--$0.21$; the arrival
counts measure this terminal precision. These finite-horizon tests do not
prove continuous-time admissibility from every initial state.
\input{generated/paired_obstacle}
\input{generated/paired_obstacle_results}

\begin{figure}[tbp]
\centering\includegraphics[width=0.78\linewidth]{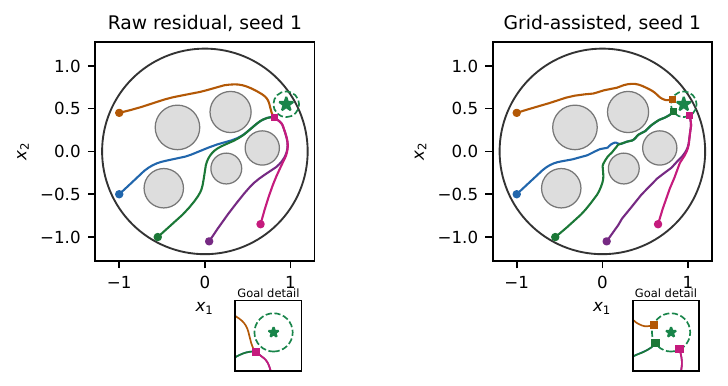}
\caption{The same five illustrative starts for the prespecified seed-one
pair at $\Delta t=0.00625$. Circles mark starts, squares stopped endpoints,
and the star the goal. Enlargements show terminal positions relative to
the goal ball. The table reports all three seeds on the same $200$ test starts.}
\label{fig:paired_obstacle}
\end{figure}

\FloatBarrier
Penalty-induced confinement yields $O(h)$ localization error with a
sufficient logarithmic margin independent of $\varepsilon$.
Convergence to the constrained value follows under the stated refinement
and evaluation-accuracy conditions. The reference calculations exhibit
controlled leakage and localization across the tested scales, including
$h=\varepsilon$. In the neural experiments, decreasing sampled residuals
can coexist with substantial value error; the paired obstacle study
connects policy-evaluation accuracy to terminal precision.

\FloatBarrier

%% file: benchmarks.tex
\subsection{Benchmark I: a cusp and a boundary correction}
\label{subsec:bench_prop511}

For $d=1$, $A=[-1,1]$, and $\lambda=1$, consider
\[
\dot x(t)=a(t),\qquad f(x,a)=e^{-|x|}.
\]
The corresponding state-dependent Hamiltonian is
\[
H(x,p)=\sup_{a\in[-1,1]}\{-ap-e^{-|x|}\}
=|p|-e^{-|x|}.
\]
For $k>0$, on $\Omega_k=(-k,k)$, Proposition~5.11 of \cite{KimTranTu2020} gives
\begin{equation}\label{eq:prop511_exact}
u_k(x)=\frac{e^{-|x|}}{2}+\frac{e^{|x|-2k}}{2},
\qquad
u_{\rm free}(x)=\frac{e^{-|x|}}{2},
\end{equation}
for the state-constraint and whole-space solutions, respectively.  Hence
\[
0\le u_k(x)-u_{\rm free}(x)=\frac{e^{|x|}}{2}e^{-2k},
\qquad x\in[-k,k],
\]
which yields exponential convergence on every fixed interior set as
$k\to\infty$.  This benchmark tests whether the numerical method resolves the
cusp at $x=0$ and the state-dependent running cost while preserving the
boundary correction in \eqref{eq:prop511_exact}.

\subsection{Benchmark II: boundary selection}
\label{subsec:bench_prop510}

For $k>0$, let $d=1$, $A=[-1,1]$, and $\lambda=1$, with
\[
\dot x(t)=a(t),\qquad f(a)=1-a.
\]
The Hamiltonian and state-constraint HJB equation are
\[
H(p)=\sup_{a\in[-1,1]}\{-ap-(1-a)\}=|p-1|-1,
\qquad u_k+H(u_k')=0\quad\text{in }(-k,k).
\]
Proposition~5.10 of \cite{KimTranTu2020} gives
\begin{equation}\label{eq:prop510_explicit}
u_k(x)=e^{x-k},\qquad x\in[-k,k],
\end{equation}
whereas the whole-space solution is $u_{\rm free}\equiv0$.  In particular,
$u_k(k)=1$ for every $k$, so convergence to the whole-space solution fails
uniformly near the moving boundary.  This is a stringent boundary-selection
test for the exterior penalty.

\subsection{Benchmark III: an exact solution on a high-dimensional cylinder}
\label{subsec:bench_highd}

For $d\ge2$ and $k>0$, define
\[
S(x)=\sum_{i=1}^d x_i,\qquad
Q(x)=\sum_{i=1}^d x_i^2-\frac{S(x)^2}{d},\qquad
r_k=\frac{k}{\sqrt{d(d-1)}},
\]
and the cylinder
\begin{equation}\label{eq:cyl_domain}
\Omega_k=\{x\in\R^d:\ |S(x)|<k,\ Q(x)<r_k^2\}.
\end{equation}
The cap--lateral corners lie outside the global $C^2$ domain hypothesis;
the direct value verification below does not require that hypothesis.
With $A=[-1,1]^d$, $\lambda=1$, and $\mathbf1=(1,\ldots,1)$, take
\[
\dot x(t)=\frac{a(t)}{d},\qquad
f(a)=\frac1d\sum_{i=1}^d(1-a_i).
\]
Writing $\|z\|_1=\sum_i|z_i|$, the separable Hamiltonian is
\begin{equation}\label{eq:cyl_hamiltonian}
H(p)=\frac1d\sum_{i=1}^d(|p_i-1|-1)
=\frac{\|p-\mathbf1\|_1}{d}-1.
\end{equation}
The exact state-constraint solution is
\begin{equation}\label{eq:cyl_exact_solution}
u_k(x)=\exp(S(x)-k),\qquad x\in\overline{\Omega}_k.
\end{equation}
The identity $Du_k=u_k\mathbf1$
verifies the equation in the interior. At
$S(x)=k$, a lower touching test function $\varphi$ satisfies
$\varphi+H(D\varphi)=\|D\varphi-\mathbf1\|_1/d\ge0$.  At the opposite cap
and on the lateral boundary, the feasible direction $\mathbf1$ gives
$D\varphi(x)\cdot\mathbf1\le Du_k(x)\cdot\mathbf1=du_k(x)$.
Together with
\[
\|D\varphi-\mathbf1\|_1\ge d-D\varphi\cdot\mathbf1
\]
this implies $u_k+H(D\varphi)\ge u_k-D\varphi\cdot\mathbf1/d\ge0$.
Thus
\eqref{eq:cyl_exact_solution} remains a valid reference even at the nonsmooth
cap--lateral intersections.

The value can also be verified directly, without a boundary comparison
theorem. The admissible control $a=\mathbf1$ until time $k-S(x)$, followed
by $a=0$, keeps $Q$ constant and has cost $e^{S(x)-k}$. Conversely, write
$S(t)=S(X^{x,a}(t))$ and $J(x,a)=\int_0^\infty e^{-t}f(a(t))\,dt$.
Every admissible trajectory satisfies $S(t)\le\min\{S(x)+t,k\}$.
Since $f=1-\dot S$, integration by parts gives
\[
 J(x,a)=1+S(x)-\int_0^\infty e^{-t}S(t)\,dt
          \ge e^{S(x)-k}.
\]
Thus the formula is the control value, including at the cap--lateral corners.

The value depends on the single coordinate $S$. These experiments test a
high-dimensional geometry and implementation with a known reference; they
do not establish accuracy for solutions with high intrinsic dimension.

%% file: generated/confinement_results.tex
Figure~\ref{fig:error_mechanism} (top) shows leakage divided by $h$ decreasing under refinement and localization errors at most $0.276h$ over all 15 pairs. The modest logarithmic margin illustrates uniform behavior over the tested penalty range; it does not test the conservative sufficient constant in the theorem. Appendix~\ref{supp:confinement_diagnostic} gives the outer-boundary and solver-error checks.

%% file: generated/mechanism_results.tex
The lower panels of Figure~\ref{fig:error_mechanism} separate the error fields. In II, the final mesh/viscosity difference and remaining-bias proxy have RMS $2.80\times10^{-3}$ and $2.62\times10^{-3}$, yet their sum has smaller RMS because the signed errors partly cancel. The final neural error also includes evaluation and policy-iteration errors.

%% file: generated/exact_compact.tex
\begin{table}[tbp]\centering\small
\caption{Relative RMS errors (\%), mean $\pm$ sample standard deviation over five paired seeds. The last column is the within-seed PI-minus-direct difference in percentage points; nominal optimization budgets are matched.}\label{tab:exact_compact}
\begin{tabular}{lrrr}\toprule Benchmark & PINN-PI & Direct Bellman & Paired difference \\\midrule
I & 1.69 $\pm$ 1.33 & 1.76 $\pm$ 0.49 & -0.07 $\pm$ 1.43 \\
II & 2.07 $\pm$ 0.78 & 2.26 $\pm$ 0.59 & -0.19 $\pm$ 1.06 \\
III, $d=2$ & 5.02 $\pm$ 0.98 & 5.38 $\pm$ 0.48 & -0.35 $\pm$ 0.79 \\
III, $d=5$ & 4.37 $\pm$ 0.63 & 5.62 $\pm$ 0.71 & -1.25 $\pm$ 1.33 \\
III, $d=10$ & 5.78 $\pm$ 1.62 & 8.05 $\pm$ 0.70 & -2.28 $\pm$ 0.95 \\
III, $d=20$ & 15.99 $\pm$ 3.20 & 15.04 $\pm$ 2.11 & 0.95 $\pm$ 4.60 \\
\bottomrule\end{tabular}\end{table}

%% file: generated/paired_obstacle.tex
\begin{table}[tbp]\centering\small\setlength{\tabcolsep}{4pt}
\caption{Paired obstacle evaluation at the final stage. Residual is domain Bellman RMS; evaluation and policy errors are domain grid maxima against the frozen-policy value and midpoint-boundary Bellman value, respectively. Arrivals and segment violations use 200 common test starts at $\Delta t=0.00625$.}\label{tab:paired_obstacle}
\begin{tabular}{lrrrrr}\toprule Evaluator, seed & Residual & Eval. error & Policy error & Arrivals & Violations \\\midrule
Raw, 1 & 0.023 & 0.070 & 0.055 & 0/200 & 0 \\
Grid-assisted, 1 & 0.161 & 0.022 & 0.029 & 200/200 & 0 \\
Raw, 2 & 0.020 & 0.057 & 0.054 & 67/200 & 0 \\
Grid-assisted, 2 & 0.135 & 0.015 & 0.015 & 200/200 & 0 \\
Raw, 3 & 0.018 & 0.059 & 0.036 & 0/200 & 0 \\
Grid-assisted, 3 & 0.135 & 0.015 & 0.008 & 200/200 & 0 \\
Finite-grid reference & --- & --- & --- & 200/200 & 0 \\
\bottomrule\end{tabular}\end{table}

%% file: generated/paired_obstacle_results.tex
Counts at all three time steps are provided in Appendix~\ref{supp:paired_obstacle}. At $\Delta t=0.00625$, across the three seeds, raw-residual evaluation reaches the goal from $0$--$67$ of the $200$ test starts, whereas grid-assisted evaluation reaches it from $200$ starts. The raw-trained networks have smaller sampled Bellman residuals, but larger frozen-policy and deployed-policy value errors and fewer goal arrivals. This paired experiment compares evaluation procedures, including their finite-grid boundary treatment.

%% file: appendices.tex
\input{policy_cost}

\section{Residual lifting}
\label{supp:l2_lifting}
\subsection{From \texorpdfstring{$L^2$}{L2} PINN residuals to value accuracy}
\label{subsec:l2_pinn_residuals_to_accuracy}

PINNs typically minimize an empirical mean-square residual. The following
estimate relates a population $L^2$ norm to the pointwise supremum under
Lipschitz regularity; controlling the population norm is a separate step.

\begin{lemma}
\label{lem:l2_to_linf_residual_lifting}
Let $D\subset\mathbb R^d$ be a bounded Lipschitz domain (for instance a
box, the only case used below). Let the pointwise function
$r:\overline D\to\R$ be Lipschitz and satisfy
\[
    \operatorname{Lip}(r;D)\le \Lambda.
\]
Then there exists a constant $C_D>0$, depending only on $D$ and $d$, such that
\begin{equation}\label{eq:l2_to_linf_residual_lifting}
    \|r\|_{L^\infty(D)}
    \le
    C_D
    \Lambda^{\frac{d}{d+2}}
    \|r\|_{L^2(D)}^{\frac{2}{d+2}}
    +
    C_D
    \|r\|_{L^2(D)}.
\end{equation}
In particular, if $\Lambda$ is uniformly bounded and
$\|r\|_{L^2(D)}\to0$, then $\|r\|_{L^\infty(D)}\to0$.
\end{lemma}

\begin{proof}
The case $M:=\|r\|_{\infty,D}=0$ is immediate; if $\Lambda=0$,
$r$ is constant and the second term suffices. Otherwise choose
$x_0\in\overline D$ with $|r(x_0)|=M$. A bounded Lipschitz domain has
constants $c_D,r_D>0$ with $|B_s(x_0)\cap D|\ge c_Ds^d$ for
$0<s\le r_D$. Set $s=\min\{M/(2\Lambda),r_D\}$.
Then $|r|\ge M/2$ on this intersection, so
$\|r\|_{L^2(D)}^2\ge c_DM^2s^d/4$. The two choices of $s$ give
the two terms in \eqref{eq:l2_to_linf_residual_lifting}.
\end{proof}

For $D=\operatorname{int}(Q_{L'}^h)$, define the weighted residual norm by
$\|r\|_{L^2(\rho_{L'})}^2:=\int_D|r(x)|^2\rho_{L'}(x)\,dx$.
The population integral in \eqref{eq:pinn_policy_eval_loss} controls this
norm; the empirical collocation average alone does not. Assuming density bounds
\[
    0<\rho_{\min,L'}\le\rho_{L'}(x)\le\rho_{\max,L'}<\infty
    \qquad\text{on }Q_{L'}^h,
\]
one has
$\|r\|_{L^2(Q_{L'}^h)}\le\rho_{\min,L'}^{-1/2}\|r\|_{L^2(\rho_{L'})}$,
so the lemma applies to the unweighted norm; if $L'\to\infty$ along a
refinement sequence, the dependence of $\rho_{\min,L'}$ and of the lifting
constant $C_D$ on the box size must be included in the smallness
assumptions. Applying Lemma~\ref{lem:l2_to_linf_residual_lifting} to
\[
    r(x)=\mathcal L_{\varepsilon,h}^{\pi_n}v_{\theta_n}(x)
\]
therefore bounds its pointwise supremum if the residual itself has the
stated Lipschitz regularity. Corollary~\ref{cor:closed_certificate}
then bounds evaluation error, including the damped outer-strip term.
Neither a small empirical loss nor membership in a Sobolev equivalence
class alone controls values on an exceptional lattice coset.

\begin{remark}[Residual regularity and discontinuous policies]
\label{rem:residual_regularity}
The Lipschitz hypothesis on the residual is an additional restriction. The map
$x\mapsto\mathcal L_{\varepsilon,h}^{\pi_n}v_{\theta_n}(x)$ involves
$x\mapsto b(x,\pi_n(x))$ and $x\mapsto f(x,\pi_n(x))$, which can be
discontinuous across the switching surfaces of a bang-bang policy
(cf.~Remark~\ref{rem:no_selector_regularity}). A switching region of small
measure need not have a small residual supremum. The lifting lemma
therefore requires regularity of the residual itself.
Alternatively, let $Z\subset\overline D$ be a finite validation set with
covering radius $\eta=\sup_{x\in D}\min_{z\in Z}|x-z|$.
A known nondecreasing modulus $\omega$ satisfying
$|r(x)-r(y)|\le\omega(|x-y|)$ on $\overline D$ gives
$\|r\|_{\infty,D}\le\max_{z\in Z}|r(z)|+\omega(\eta)$.
This controls the supremum directly from pointwise samples.
\end{remark}


\section{Additional residual and selector observations}
\label{supp:evaluation_certificates}
This section gives a same-box evaluation estimate and explains the role
of selector discontinuities in greedy improvement.

\subsection{Estimating value mismatch by policy-evaluation residuals}
\label{subsec:value_mismatch_by_policy_eval_residual}

This same-box estimate retains an unknown boundary-strip mismatch.
The closed outer-box estimate \eqref{eq:closed_certificate} avoids this term.

For a policy $\pi$, define the boundary-strip evaluation mismatch
\begin{equation}\label{eq:boundary_strip_eval_mismatch}
    \mathfrak B_{\varepsilon,h,L}^{\pi}(v)
    :=
    \left\|
        v-S_{\varepsilon,h}(\pi)
    \right\|_{L^\infty(\partial_h Q_L)}.
\end{equation}

\begin{proposition}
\label{prop:value_mismatch_by_residual}
For a bounded measurable policy $\pi$ and bounded candidate value $v$,
\begin{equation}\label{eq:value_mismatch_by_residual}
    \left\|
        v-S_{\varepsilon,h}(\pi)
    \right\|_{L^\infty(Q_L)}
    \le
    \max
    \left\{
        \mathfrak B_{\varepsilon,h,L}^{\pi}(v),
        \frac{1}{\lambda}
        \mathfrak E_{\varepsilon,h,L}^{\infty}(v,\pi)
    \right\}.
\end{equation}
Consequently,
\begin{equation}\label{eq:xi_val_residual_bound}
    \xi_n^{\rm val}
    \le
    \max
    \left\{
        \mathfrak B_{\varepsilon,h,L}^{\pi_n}(v_{\theta_n}),
        \frac{1}{\lambda}
        \mathfrak E_{\varepsilon,h,L}^{\infty}
        (v_{\theta_n},\pi_n)
    \right\}.
\end{equation}
\end{proposition}

\begin{proof}
Set $z=v-S_{\eps,h}(\pi)$, $r=\mathcal L^\pi_{\eps,h}v$, and
$B=\|z\|_{\infty,\partial_hQ_L}$. If $\sup_{Q_L}z>B$, choose
$x_\eta\in Q_L^h$ with $z(x_\eta)\ge\sup_{Q_L}z-\eta$.
Every neighbor lies in $Q_L$, so the argument of
Lemma~\ref{pc:lem:comparison} gives
$\lambda(\sup_{Q_L}z-\eta)\le\|r\|_{\infty,Q_L^h}+\Lambda_h\eta$.
Let $\eta\downarrow0$ and repeat for $-z$. The case in which the
supremum is bounded by $B$ is immediate.
\end{proof}

\begin{remark}
\label{rem:boundary_strip_localization}
Corollary~\ref{cor:closed_certificate} controls the unknown strip mismatch
by validating on a nested outer box. For an interior target
$K\subset\overline\Omega\Subset Q_L^h$ and
$\lambda h\le2dN$,
Proposition~\ref{prop:interior_certificate} bounds its influence by the
factor $\beta_{h,L}\le e^{-\lambda d_K/(4dN)}$,
where $d_K=\dist(K,\partial_hQ_L)$ and
$\beta_{h,L}=\gamma_h^{\lceil d_K/h\rceil}$. The strip term
itself is bounded a priori by the observable quantity
\[
    \mathfrak B_{\varepsilon,h,L}^{\pi}(v)
    \le
    \|v\|_{L^\infty(\partial_h Q_L)}
    +
    \frac{\|f\|_\infty+M_p/\varepsilon}{\lambda},
\]
by Lemma~\ref{pc:lem:representation}. The margin geometry gives
$d_K\ge D_L$, so under \eqref{pc:eq:margin_condition} and
$h\le\varepsilon$, the a priori part satisfies
\[
    \beta_{h,L}\,\frac{\|f\|_\infty+M_p/\varepsilon}{\lambda}
    =O(h),
\]
while the network contribution
$\beta_{h,L}\|v\|_{L^\infty(\partial_hQ_L)}$ is $O(h)$ as well provided
the strip amplitude is uniformly bounded. For polynomial growth in
$1/h$ and $1/\varepsilon$, the logarithmic-margin coefficient must be
increased according to the growth exponents so that the product of the
strip amplitude and $\beta_{h,L}$ is $O(h)$. This gives an interior
evaluation bound with an explicitly controlled remainder.
\end{remark}

\subsection{Selector discontinuities and ties}

\begin{remark}
\label{rem:no_selector_regularity}
Requiring the greedy control to depend Lipschitz continuously on the
gradient excludes the bang-bang benchmark of
Section~\ref{subsec:bench_prop510}. For $H(p)=|p-1|-1$, the maximizing
control jumps between $\pm1$ as $p$ crosses $1$. The greedy-gap estimate
instead compares Hamiltonian values, which remain Lipschitz in the
gradient even when the selected controls jump.
\end{remark}

\begin{remark}
\label{rem:ties}
Two policies can be exactly greedy for the same value and yet have
different evaluated values. Thus a zero greedy gap alone cannot control
the difference between two improved policy values.

To see this directly, take the finite-lattice operator with $\lambda=h=1$,
$N=1/2$, controls $a\in\{-1,0,1\}$, drift $b=a$, interior nodes
$-1,0,1$, and zero boundary values at $\pm2$. The running costs,
with rows in node order and columns in control order, are
\[
 (c(x,a))_{x,a}=\begin{pmatrix}2&6&20\\2&3&0\\20&10&4\end{pmatrix}.
\]
The policy $a=0$ has value $w=(4,4,6)$. Its two greedy improvements
$(-1,-1,1)$ and $(-1,1,1)$ both have zero gap, but their evaluated
values are $(1,3/2,2)$ and $(1,1,2)$, respectively.
Their supremum difference is $1/2$, so even this finite-lattice case
precludes a bound proportional only to the greedy gap.

Proposition~\ref{prop:assumption_free_api} instead compares each new
policy value with the fixed point $v_L^*$. Its proof uses the one-sided
comparison $v_L^*\le w_{n+1}$ and contraction of the localized Bellman
operator, without comparing alternative greedy selectors.
\end{remark}


\section{Refinement from validated residuals}
\label{supp:refinement_certificates}
The evaluation condition in (R4) can be expressed through residual and
strip suprema via the nested-box certificate. These quantities require
validated bounds, not only samples.

\begin{corollary}
\label{cor:observable_refinement}
In the setting of
Corollary~\ref{cor:local_convergence_from_mismatch_control}, let each
residual be validated on an outer training box $Q_{L'_k}$ with
$Q_{L_k}\Subset Q_{L'_k}^{h_k}$ and margin
$d_k:=\dist(Q_{L_k},\partial_{h_k}Q_{L'_k})$, and set
\[
    \mathcal E_k
    :=
    \sup_{0\le j\le n_k}
    \bigl\|\mathcal L^{\pi_{k,j}}_{\varepsilon_k,h_k}v_{\theta_{k,j}}
    \bigr\|_{L^\infty(Q_{L'_k}^{h_k})},
    \qquad
    \mathcal A_k
    :=
    \sup_{0\le j\le n_k}
    \|v_{\theta_{k,j}}\|_{L^\infty(\partial_{h_k}Q_{L'_k})}.
\]
If (R1), (R2), (R3), and (R5) hold and
\begin{gather*}
    \frac{1}{h_k^{2}}
    \Bigl[
        \frac{\mathcal E_k}{\lambda}
        +e^{-\lambda d_k/(4dN)}
        \Bigl(\mathcal A_k
        +\frac{\|f\|_\infty+M_p/\varepsilon_k}{\lambda}\Bigr)
    \Bigr]
    \longrightarrow0,
    \\
    \frac{1}{h_k}\sup_{0\le j<n_k}
    \mathfrak G^{\infty}_{\varepsilon_k,h_k,L_k}
    (v_{\theta_{k,j}},\pi_{k,j+1})\longrightarrow0,
\end{gather*}
then the conclusions of
Corollary~\ref{cor:local_convergence_from_mismatch_control} hold. 
\end{corollary}

\begin{proof}
By Corollary~\ref{cor:closed_certificate} (applicable since
$\lambda h_k\le2dN$ by (R1)), $\xi_{k,j}^{\rm val}$ is bounded by the
bracketed expression for every $j\le n_k$, so the displayed conditions
imply (R4).
\end{proof}

The first condition uses residual and strip suprema together with known
problem constants. A finite validation net bounds these suprema only
when its covering radius and a pointwise modulus of continuity are also
controlled, as in Remark~\ref{rem:residual_regularity}.

\begin{corollary}[Refinement from Lipschitz residuals and population $L^2$ bounds]
\label{cor:convergence_from_l2_pinn_residuals}
Use the setting, margins, and strip amplitudes of
Corollary~\ref{cor:observable_refinement}. Put
$D_k=\operatorname{int}(Q_{L'_k}^{h_k})$. Suppose each pointwise residual
\[
    r_{k,j}(x)
    :=
    \mathcal L_{\varepsilon_k,h_k}^{\pi_{k,j}}
    v_{\theta_{k,j}}(x),\qquad 0\le j\le n_k,
\]
is Lipschitz on $\overline D_k$ with constant at most $\Lambda_k$.
Assume the sampling density is bounded below by $\rho_{\min,k}>0$,
and let
\[
 \ell_k=\sup_{0\le j\le n_k}
   \left(\int_{D_k}|r_{k,j}(x)|^2\rho_{L'_k}(x)\,dx\right)^{1/2},
 \qquad a_k=\rho_{\min,k}^{-1/2}\ell_k.
\]
Define
\[
 E_k^{(2)}=C_{D_k}\left(
      \Lambda_k^{d/(d+2)}a_k^{2/(d+2)}+a_k\right).
\]
If (R1)--(R3), (R5), and the improvement-gap condition in (R4) hold, and
\[
 \frac1{h_k^2}\left[
      \frac{E_k^{(2)}}{\lambda}
      +e^{-\lambda d_k/(4dN)}
        \left(\mathcal A_k+M_{\varepsilon_k}\right)\right]\longrightarrow0,
\]
then $v_{\theta_{k,n_k}}\to u$ uniformly and in $L^2$ on $K$.
\end{corollary}

\begin{proof}
The density lower bound gives $\|r_{k,j}\|_{L^2(D_k)}\le a_k$.
Lemma~\ref{lem:l2_to_linf_residual_lifting} gives
$\mathcal E_k\le E_k^{(2)}$; continuity includes the boundary of $D_k$.
Apply Corollary~\ref{cor:observable_refinement}.
\end{proof}
The displayed condition accounts explicitly for the growth of the box,
the density factor, and the residual Lipschitz constant. It assumes a
bound on the population integral defining $\ell_k$, which must be
justified separately from an empirical training average.


\input{numerical_details}

%% file: policy_cost.tex
\section{An auxiliary cost bound along policy iteration}
\label{supp:policy_cost}
The following estimate bounds the costs of individual implemented policies
by the initial leakage cost and accumulated greedy gaps.

\begin{proposition}
\label{pc:prop:class_stability}
Assume the hypotheses of Lemma~\ref{pc:lem:leakage} and $0<h\le h_0$. Let
$\{\pi_n\}\subset\Pi_L$ be the localized PINN-PI policies starting from
$\pi_0=\bar\pi$, let
$w_n:=S_{\eps,h}(\pi_n)$, and set
\[
    g_n:=\sup_{Q_L^h}\Gamma_{\eps,h}(w_n,\pi_{n+1}).
\]
Then, for every $n\ge0$,
\begin{equation}
\label{pc:eq:class_stability}
    \sup_{\overline\Omega} w_n
    \ \le\
    \underbrace{\frac{\|f\|_\infty}{\lambda}+\frac{C_*h}{\eps}}_{=:C_0^{\rm def}}
    \;+\;
    \frac1\lambda\sum_{j=0}^{n-1}g_j
    \ =:\ C_n .
\end{equation}
In particular $\sup_K w_n\le C_n$, which is the bound required by the
confinement estimates of Section~\ref{pc:sec:confinement}.
If, in addition, $\pi_{n+1}$ is exactly greedy on $Q_L^h$ with respect to
the neural value $v_{\theta_n}$ and
$\xi_n^{\rm val}:=\|v_{\theta_n}-w_n\|_{L^\infty(Q_L)}$, then
$g_n\le(2\sqrt d\,M_b/h)\,\xi_n^{\rm val}$, so that
\begin{equation}
\label{pc:eq:class_stability_val}
    C_n
    \ \le\
    C_0^{\rm def}
    +\frac{2\sqrt d\,M_b}{\lambda\,h}\sum_{j=0}^{n-1}\xi_j^{\rm val}.
\end{equation}
\end{proposition}

\begin{proof}
The base case is Lemma~\ref{pc:lem:leakage}:
$w_0=S_{\eps,h}(\bar\pi)\le C_0^{\rm def}$ on $\overline\Omega$.
For the inductive step, let
$z:=w_{n+1}-w_n$.  As in Lemma~\ref{pc:lem:no_exterior_source},
$\mathcal A^{\pi_{n+1}}z=-\mathcal L^{\pi_{n+1}}_{\eps,h}w_n$, the source
vanishes on $\R^d\setminus Q_L^h$ (both policies equal $\bar\pi$ there and
$\mathcal L^{\pi_n}_{\eps,h}w_n=0$), while on $Q_L^h$
\[
    -\mathcal L^{\pi_{n+1}}_{\eps,h}w_n
    =
    -\bigl(\mathcal L^{\pi_{n+1}}_{\eps,h}w_n
    -\mathcal L^{\pi_n}_{\eps,h}w_n\bigr)
    =
    \Gamma_{\eps,h}(w_n,\pi_{n+1})-\Gamma_{\eps,h}(w_n,\pi_n)
    \ \le\ g_n ,
\]
using $\Gamma_{\eps,h}\ge0$.  Lemma~\ref{pc:lem:comparison} gives
$\sup_{\R^d}(w_{n+1}-w_n)\le g_n/\lambda$, and
\eqref{pc:eq:class_stability} follows by summation.  The bound on $g_n$
under exact hard improvement follows from
Lemma~\ref{lem:greedy_gap_transfer}, since
$\Gamma_{\eps,h}(v_{\theta_n},\pi_{n+1})=0$ on $Q_L^h$.
\end{proof}

\begin{remark}
\label{pc:rem:class_condition}
Proposition~\ref{pc:prop:class_stability} gives a sufficient condition
under which the PI-generated policies remain uniformly bounded-cost, hence
subject to the confinement estimates: under exact greedy improvement, with $h\le\eps$ and
$\sum_j\xi_j^{\rm val}\le c\,h$ for a fixed $c>0$, one has
$\sup_K w_n\le C_{\rm bd}:=\|f\|_\infty/\lambda+C_*
+2\sqrt d\,M_b\,c/\lambda$ for all $n$, uniformly in
$(\eps,h,L)$. At mesh size $h$, write $\xi_{h,j}^{\rm val}$ for the
evaluation error at iteration $j$. For $n_h$ improvements, it suffices that
$\sum_{j<n_h}\xi_{h,j}^{\rm val}\le c\,h$; a uniform sufficient
condition is $n_h\sup_{j<n_h}\xi_{h,j}^{\rm val}=O(h)$. Since
$(1-\gamma_h)/h\to\lambda/(2dN)>0$ and the initial global error is $O(h^{-1})$ when
$h\le\eps$, choosing $n_h$ proportional to $|\log h|/h$ with a
sufficiently large constant makes the initial-error bound a power
of $h$. In this regime a sufficient uniform evaluation
accuracy is $\sup_{j<n_h}\xi_{h,j}^{\rm val}=O(h^2/|\log h|)$.
This condition gives a confinement interpretation for the individual
PI policies. The main localization theorem and discounted recursion
do not require this additional summability condition.
\end{remark}


%% file: numerical_details.tex
\section{Numerical protocol and reproducibility}
\label{supp:numerical_protocol}

Section~\ref{sec:experiments} reports final-stage results from all prescribed seeds.
The code and recorded artifacts are available at
\url{https://github.com/yeoneung/state_constraint_discrete}
(release \texttt{experiments-2026-09-13}).
The manifest \path{reproduction/aligned/PAPER_RUNS.json}
identifies each run, its configuration, saved
evaluation-policy/network pairs, source snapshot, and diagnostic outputs.

\paragraph{Shared operator and networks}
The default policy is $-\tanh x$ in one dimension and
$-x/\max\{1,|x|\}$ for the cylinder. For obstacles, disjoint boundary
collars point away from holes and toward the workspace center. Their width
is one quarter of the minimum boundary separation or hole radius; the
implementation gives the continuous extension away from the collars.
Every improved policy equals this same default outside $Q_L^h$.
Writing $z$ for the scalar MLP output before the value head, the head is
$B\tanh(z/B)$ with $B=2M_\varepsilon$; obstacle
networks apply softplus to $z$ first. Cylindrical inputs augment the full
coordinates by $S/k$ and $\log(1+Q/r_k^2)$ without imposing invariance.
Obstacle inputs append the normalized signed distances to the holes and
outer circle, and the Euclidean distance to the goal, to the two coordinates.
No analytic-value target, radial-invariance loss or prescribed goal value is
used. The grid-assisted evaluator uses computed frozen-policy values.

\paragraph{Optimization and sampling}
The exact-benchmark sampler allocates $40\%$ of points to the domain,
$30\%$ to boundary shells, $20\%$ to a nearby exterior region and
$10\%$ to the larger box. Samples refresh every 25 Adam updates.
The initial learning rate is $10^{-3}$ with cosine decay to one tenth
within each stage; Adam gradients are clipped at norm 10. L-BFGS uses
strong-Wolfe search and history size 30. The stated batch sizes apply to
Adam; each L-BFGS phase uses a fresh fixed batch of twice that size.
For the exact benchmarks, its gradient and change tolerances are
$10^{-9}$ and $10^{-12}$. Its actual function evaluations
are recorded, so a nominal iteration budget is not confused with an equal
number of objective evaluations. Paired methods use the same initial
weights and sampling seeds. Analytic references never enter optimization.
\input{generated/protocol_table}

\paragraph{Validation and resources}
One-dimensional errors use 10,000 equally spaced domain points; cylinder
errors use 20,000 uniformly sampled domain points, including transverse
directions. Residual validation and all final table metrics use float64.
Learning curves use the training precision and the same evaluation points
at every stage; these points are separate from the training samples.
Residual sample maxima are not continuum supremum certificates. One-dimensional
reference checks cover all nodes of one lattice coset, using double precision
without outward rounding. Computations used an RTX 5070 GPU and an
i9-14900KF CPU. CPU/GPU placement and concurrency were selected by measured
throughput; timing records include the concurrent queue context.

\section{Policy and approximation diagnostics}
\label{supp:mechanism_details}

\paragraph{Leakage and localization across penalty scales}
\label{supp:confinement_diagnostic}
The leakage and localization diagnostic uses Benchmark II on $[-2,2]$ with
$N=4$, $\bar\pi(x)=-\tanh x$, and every $h\le\varepsilon$ pair from
$\{0.04,0.02,0.01,0.005,0.0025\}$. The larger viscosity coefficient
makes leakage and localization visible over this mesh range.
All 15 configurations use $L=2+0.06+0.02\log(0.04/h)$ and outer radius
$40$. The unrestricted solve optimizes at every interior outer-grid node;
the localized solve fixes $\bar\pi$ outside $Q_L^h$.
Both use boundary data $0$ and $M_\varepsilon=2+0.09/\varepsilon$,
and reported values are their midpoints. The pure leakage equations use
sources $p(x)=\min\{\dist(x,[-2,2]),0.3\}$ and $q(x)=p(x)^2$ without
division by $\varepsilon$, with endpoint data $0/0.3$ and $0/0.09$,
respectively. Their contribution to penalized cost is
$\widehat w_p/\varepsilon$ or $\widehat w_q/\varepsilon$.
Each source equation is solved by a banded linear solver; nonlinear
values use finite-lattice policy iteration with the same centered operator.
For each localization difference we add both computed boundary-bracket
widths and the two Bellman residual suprema divided by $\lambda$.
\input{generated/confinement_details}
The corner checks use $(h,\varepsilon)=(0.04,0.04)$,
$(0.0025,0.04)$, and $(0.0025,0.0025)$.
The modest margin is an empirical logarithmic rule; these finite-scale
observations establish neither the optimal margin nor a sharp leakage rate.

\paragraph{Neural policy errors and the signed decomposition}
For every one-dimensional saved PINN-PI stage, the independent banded
solver computes the frozen-policy value and the value of the newly greedy
policy with common zero Dirichlet data. The recursion error is the maximum
over the complete finite grid, and $g_n$ is the true greedy gap on the
improvement-grid nodes. Thus these diagnostics concern the same errors as
the finite-box version of the policy-iteration recursion. They do not infer
policy contraction from the neural loss. Across the 40 transitions of each
benchmark, the median bound-to-error ratios are 6.7 and 14.1 for I and II;
the largest ratios are approximately 2449 and 85, respectively. The bound
is sufficient and can be very conservative.
\begin{figure}[tbp]
\centering\includegraphics[width=\linewidth]{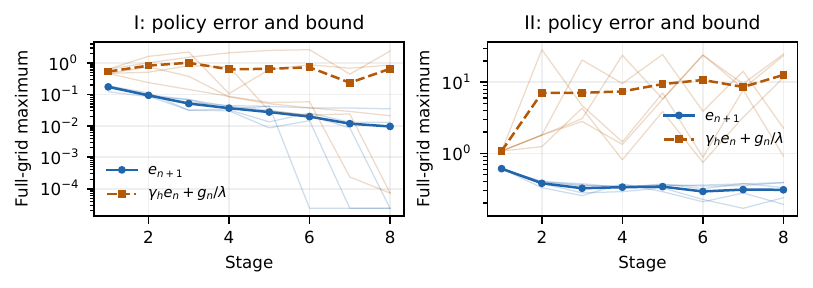}
\caption{The finite-grid policy-error recursion for every saved stage.
Thin curves show the five seeds and thick curves their means; both
benchmarks use common zero Dirichlet data.}
\label{fig:policy_recursion}
\end{figure}

For the signed decomposition in \eqref{eq:numerical_decomposition}, all four fields are
evaluated on the same 267 domain nodes of the $h=0.015$ lattice. The outer
grid endpoints are $\pm7.005$; common midpoint Dirichlet data are used for
the Bellman and policy values. Nested refinement preserves these endpoints,
the penalty and both boxes while reducing $h$ and $\nu_h$ together. Starting
with refinement factor 16, the factor doubles until the last mesh halving
changes the remaining-bias field by at most $2\%$ of its RMS. This fixed
reference-accuracy criterion gives factors 1024 and 64 for I and II,
respectively. Their successive differences are $1.07\times10^{-6}$ and
$3.25\times10^{-5}$. This is a refinement diagnostic, not a rigorous error
enclosure. The remaining-bias field includes penalty, localization and
unresolved fine-mesh effects. The signed fields close the identity to
floating-point precision; cancellation between these fields makes their
RMS norms nonadditive.
\input{generated/mechanism_details}

Separate parameter sweeps measure the sensitivity of the total
constrained-value error. In particular, a mesh sweep at fixed
penalty need not approach zero error or even decrease monotonically: the
penalty bias remains and can cancel discretization error. The localization
and outer-boundary effects in these configurations are already small;
the plots do not establish sharp localization rates.
\input{generated/reference_results}
\begin{figure}[tbp]
\centering\includegraphics[width=\linewidth]{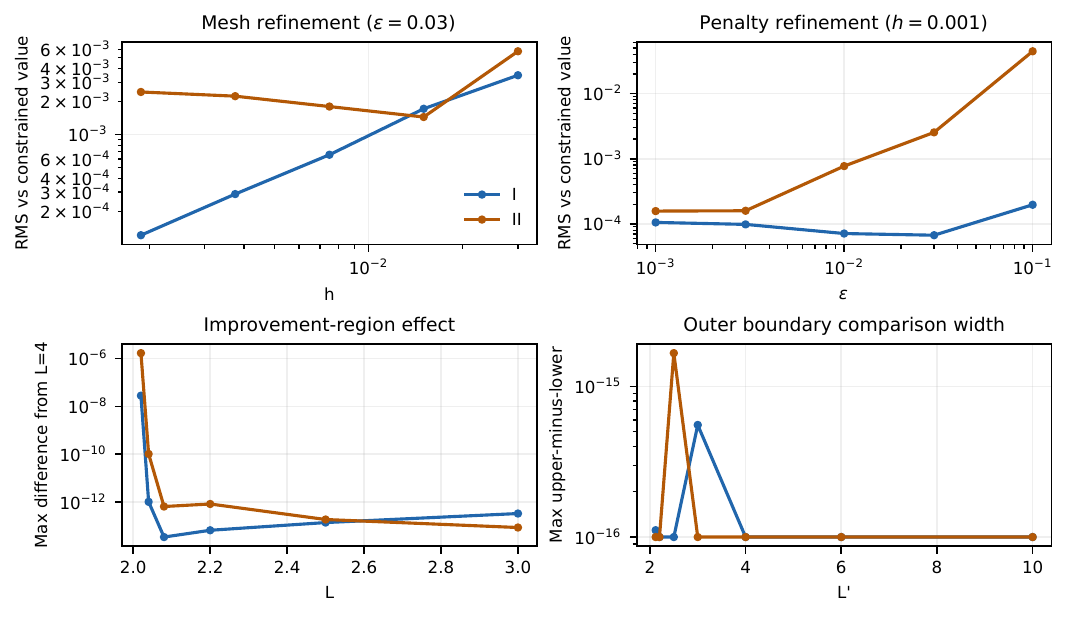}
\caption{Additional one-dimensional reference sweeps. The top panels show
total RMS error against the constrained solution, not isolated penalty or
mesh errors. The bottom panels measure improvement-box sensitivity and
outer comparison width; widths below $10^{-16}$ are plotted at that floor.}
\label{fig:reference_studies}
\end{figure}

\section{Additional explicit-benchmark results}

The following tables report absolute errors, sampled maxima, and residuals
for the explicit benchmarks; sampled maxima do not certify continuum suprema.
\input{generated/exact_results}
\input{generated/baseline_results}
\begin{figure}[tbp]
\centering\includegraphics[width=\linewidth]{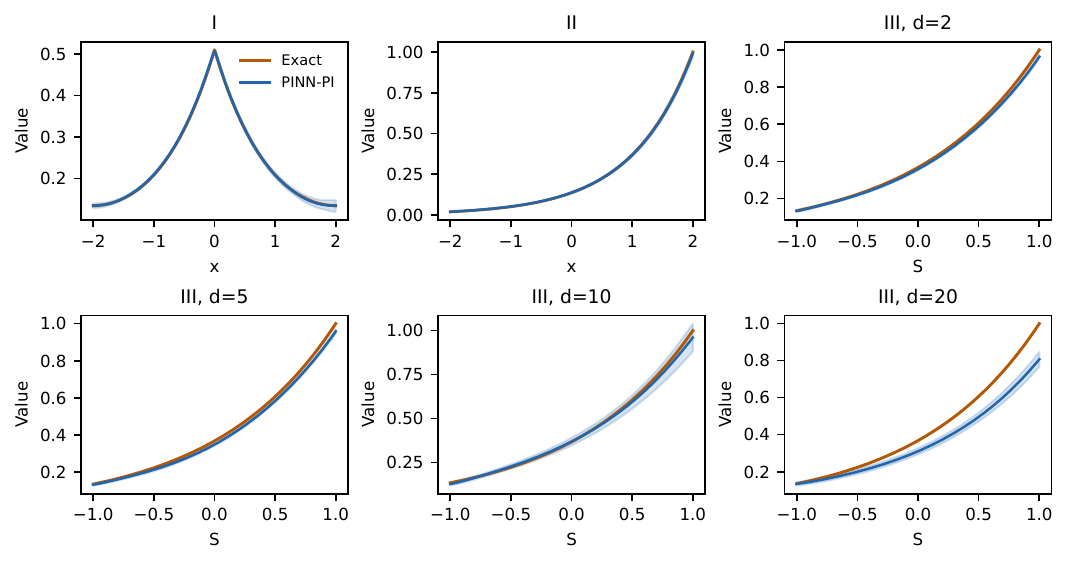}
\caption{Exact values and final PINN-PI predictions. Curves are seed means,
with one sample standard deviation shaded. Cylinder curves are axial slices;
the table metrics use the full domain.}
\label{fig:exact_solutions}
\end{figure}
\begin{figure}[tbp]
\centering\includegraphics[width=\linewidth]{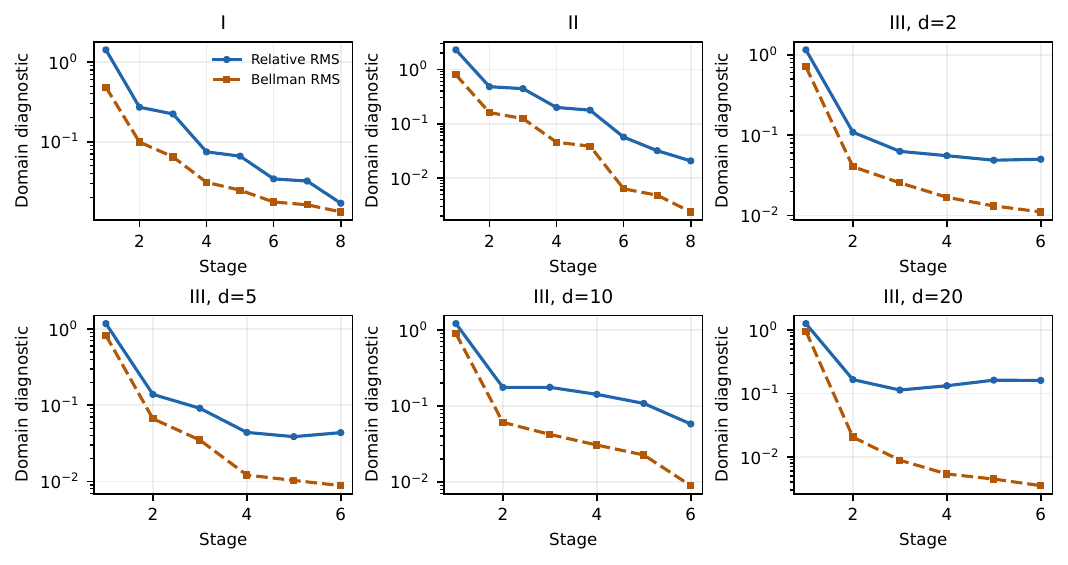}
\caption{Domain errors and Bellman residuals during PINN-PI,
averaged over all five seeds. The twenty-dimensional error can increase
while its residual decreases. Evaluation points are separate from training
samples, and the residuals are finite-sample diagnostics.}
\label{fig:aligned_learning}
\end{figure}

\section{Paired obstacle protocol and additional outcomes}
\label{supp:paired_obstacle}

The five holes have centers $(-0.32,0.28)$, $(0.25,-0.20)$,
$(-0.48,-0.43)$, $(0.30,0.46)$ and $(0.67,0.04)$, and radii
$0.26$, $0.18$, $0.23$, $0.24$ and $0.20$, respectively.
The goal-cost coefficient in Section~\ref{subsec:obstacle_navigation} is exactly
$c_g=385/512=0.751953125$. The exact greedy control inside the improvement box is
\[
 a_v(x)=\operatorname{proj}_{\overline B_1(0)}
       \left(-\frac{1.25}{0.04}\nabla_hv(x)\right).
\]
Here $\operatorname{proj}_{\overline B_1(0)}(y)=y/\max\{1,|y|\}$ is
Euclidean projection onto the closed unit ball.
The sampling mixture is $25\%$ uniform free points, $15\%$ free points
near the goal (Gaussian scale $0.2$, rejected outside the free region),
$30\%$ boundary shells, $20\%$ outer annulus points and $10\%$ uniform
box points. The near-goal points carry residuals or computed policy values,
not a prescribed value. Both paired arms use float32 training with TF32
disabled; all grid solves, table diagnostics and rollouts use float64.
L-BFGS tolerances are $10^{-8}$ for the gradient and $10^{-11}$ for change,
identically in both arms. Equal nominal budgets do not imply equal actual
line-search counts or wall time.

The two arms share initial-model hashes, per-stage sampling-generator
states and recorded batch hashes. Only the evaluation objective changes.
The grid-assisted arm fits the midpoint of the $0/M_\varepsilon$ frozen-policy
solutions, with bilinear interpolation for off-grid targets. This equals
the finite-grid policy value with boundary $M_\varepsilon/2$. The raw arm
uses the whole-space network residual; its grid solves are
diagnostics and do not enter its loss. The comparison therefore changes
the evaluation procedure, including boundary information and numerical
preconditioning. It does not isolate one scalar error measure as the sole
cause of navigation performance. The grid reference already supplies a
feedback; here the network is a learned value representation whose greedy
feedback is evaluated. Sparse policy solves are additional training work
for the grid-assisted arm.

The three paired seeds and the protocol were fixed in the run manifest
before training. The held-out set contains 200 common starts generated
with seed 20260911 and excludes the goal ball. We report all six controllers.
The initial states are shared across seeds and step sizes, so the tests are
not independent samples of initial states across controllers. Each Euler
segment is checked analytically against the circular holes and outer disk.
Strict penetration uses tolerance $10^{-12}$; a separate $10^{-4}$-tolerance
diagnostic is also stored. Termination is at first entry into the goal ball
or horizon 20. All paths, including failures, are saved.
\input{generated/paired_obstacle_details}
\FloatBarrier

The raw-arm failures settle near the goal but outside its prescribed ball:
their terminal distances are approximately $0.209$, $0.206$ and $0.176$
for seeds 1, 2 and 3. Thus the arrival counts concern the stated radius
$0.15$, and should not be read as an inability to traverse the free region.

%% file: generated/protocol_table.tex
\begin{table}[tbp]\centering\small\setlength{\tabcolsep}{3pt}
\caption{Exact-benchmark protocol, common to paired methods. All MLPs have three hidden layers. Adam uses batches of $1024$; L-BFGS uses a fresh fixed batch of $2048$. The last column is the L-BFGS iteration budget per stage.}\label{tab:aligned_protocol}
\begin{tabular}{lrrrrrr}\toprule Benchmark & $(\varepsilon,h)$ & $\kappa$ & $(L,L')$ & Width & Adam updates & L-BFGS \\\midrule
I & (0.03,0.015) & 1 & (4,7) & 64 & $8\times 1600$ & 100  \\
II & (0.03,0.015) & 1 & (4,7) & 64 & $8\times 1600$ & 100  \\
III, d=2 & (0.02,0.008) & 1 & (2,4) & 96 & $6\times 600$ & 30  \\
III, d=5 & (0.02,0.008) & 1 & (2,4) & 96 & $6\times 600$ & 30  \\
III, d=10 & (0.02,0.006) & 1 & (2,4) & 96 & $6\times 600$ & 30  \\
III, d=20 & (0.02,0.006) & 1 & (2,4) & 96 & $6\times 600$ & 30  \\
\bottomrule\end{tabular}\end{table}
One-dimensional training uses float64 on CPU (eight threads per run); the cylindrical runs use float32 on GPU. All table diagnostics are reevaluated in float64. TF32 is disabled. The learning-rate floor is one tenth of the stage's initial rate. No stage is selected by test error.

%% file: generated/confinement_details.tex
The maximum combined boundary-width and residual estimate is $1.63\times10^{-10}$, and it is less than $0.07\%$ of the measured localization difference in every case. Increasing the outer radius from $40$ to $60$ at the three declared corner configurations changes the localization difference by at most $1.14\times10^{-14}$. These are double-precision checks without outward rounding.

%% file: generated/mechanism_details.tex
\begin{table}[tbp]\centering\small\setlength{\tabcolsep}{4pt}
\caption{Final signed error-field RMS, averaged over the five PINN-PI seeds. The norms are not additive. The last column is a successive fine-mesh difference, not a rigorous continuum error bound.}\label{tab:mechanism_details}
\begin{tabular}{lrrrrrr}\toprule & Evaluation & PI & Mesh/viscosity & Bias proxy & Total & Fine change \\\midrule
I & $6.51\times10^{-3}$ & $4.15\times10^{-3}$ & $1.78\times10^{-3}$ & $6.75\times10^{-5}$ & $4.53\times10^{-3}$ & $1.07\times10^{-6}$ \\
II & $4.88\times10^{-3}$ & $3.08\times10^{-3}$ & $2.80\times10^{-3}$ & $2.62\times10^{-3}$ & $7.32\times10^{-3}$ & $3.25\times10^{-5}$ \\
\bottomrule\end{tabular}\end{table}

%% file: generated/reference_results.tex
The reference study contains 62 configurations. Along the coupled sequence $h=\varepsilon/2$, reducing $\varepsilon$ from $0.1$ to $0.003$ changes the two RMS errors from $4.95\times10^{-3}$--$4.30\times10^{-2}$ to $1.45\times10^{-4}$--$1.72\times10^{-4}$. Separate mesh and penalty sweeps are shown in Figure~\ref{fig:reference_studies}. Six deliberately inexact evaluations use perturbation amplitudes $0.001$, $0.01$, and $0.05$; all 174 recorded recursion checks have nonnegative numerical slack (minimum $2.22\times10^{-14}$). The localization plot compares finite-lattice values with the largest improvement box; the outer-box plot reports the independently computed comparison width. Widths below $10^{-16}$ are displayed at the plotting floor.

%% file: generated/exact_results.tex
\begin{table}[tbp]\centering\small\setlength{\tabcolsep}{3pt}
\caption{Final PINN-PI errors, mean $\pm$ sample standard deviation over five seeds. All metrics use double-precision evaluation at domain points separate from training samples.}\label{tab:aligned_exact}
\begin{tabular}{lrrr}\toprule Benchmark & RMS & Relative RMS (\%) & Sample maximum \\\midrule
I & 0.005 $\pm$ 0.004 & 1.69 $\pm$ 1.33 & 0.011 $\pm$ 0.008 \\
II & 0.007 $\pm$ 0.003 & 2.07 $\pm$ 0.78 & 0.021 $\pm$ 0.009 \\
III, d=2 & 0.025 $\pm$ 0.005 & 5.02 $\pm$ 0.98 & 0.072 $\pm$ 0.020 \\
III, d=5 & 0.022 $\pm$ 0.003 & 4.37 $\pm$ 0.63 & 0.057 $\pm$ 0.008 \\
III, d=10 & 0.029 $\pm$ 0.008 & 5.78 $\pm$ 1.62 & 0.068 $\pm$ 0.022 \\
III, d=20 & 0.080 $\pm$ 0.016 & 15.99 $\pm$ 3.20 & 0.183 $\pm$ 0.033 \\
\bottomrule\end{tabular}\end{table}

%% file: generated/baseline_results.tex
\begin{table}[tbp]\centering\small\setlength{\tabcolsep}{3pt}
\caption{Matched nominal optimization budgets: PINN-PI and direct Bellman residual minimization. Entries are seed means; detailed timing records are supplied in the machine-readable tables.}\label{tab:aligned_baseline}
\begin{tabular}{lrrrr}\toprule & \multicolumn{2}{c}{Relative RMS (\%)} & \multicolumn{2}{c}{Bellman RMS} \\ Benchmark & PI & Direct & PI & Direct \\\midrule
I & 1.689 & 1.759 & 0.013 & 0.013 \\
II & 2.068 & 2.256 & 0.002 & 0.003 \\
III, d=2 & 5.025 & 5.379 & 0.011 & 0.011 \\
III, d=5 & 4.371 & 5.617 & 0.009 & 0.008 \\
III, d=10 & 5.778 & 8.054 & 0.009 & 0.005 \\
III, d=20 & 15.989 & 15.040 & 0.004 & 0.004 \\
\bottomrule\end{tabular}\end{table}
The comparison does not show a uniform advantage for either method. Actual objective evaluations, elapsed times, all seed results, and finite-sample maxima are supplied in the machine-readable tables. Small residuals for a fixed penalized operator alone do not remove its penalty and discretization bias.

%% file: generated/paired_obstacle_details.tex
\begin{table}[tbp]\centering\small
\caption{All paired obstacle test outcomes. Triples correspond to $\Delta t=0.025,0.0125,0.00625$; every entry uses the same 200 starts. Times are measured per run while sharing the machine; grid time is diagnostic overhead for the raw arm.}\label{tab:paired_obstacle_details}
\begin{tabular}{lrrrr}\toprule Evaluator, seed & Arrivals & Violations & Neural time (s) & Grid time (s) \\\midrule
Raw, 1 & 0/0/0 & 0/0/0 & 272.4 & 13.2 \\
Grid-assisted, 1 & 200/200/200 & 0/0/0 & 117.0 & 13.4 \\
Raw, 2 & 68/68/67 & 0/0/0 & 268.9 & 14.6 \\
Grid-assisted, 2 & 200/200/200 & 0/0/0 & 129.2 & 15.9 \\
Raw, 3 & 0/0/0 & 0/0/0 & 218.5 & 13.4 \\
Grid-assisted, 3 & 200/200/200 & 0/0/0 & 126.5 & 13.9 \\
\bottomrule\end{tabular}\end{table}

%% file: declarations.tex
\section*{Declarations}
\addcontentsline{toc}{section}{Declarations}
OpenAI ChatGPT and Codex were used to assist with research and manuscript preparation.
The authors assume responsibility for all content.